\documentclass[11pt]{article}
\usepackage{mathrsfs}
\usepackage{graphicx}
\usepackage{amssymb}
\usepackage{amscd}
\usepackage{amsfonts,amsmath}
\usepackage{float,color}
\usepackage{appendix}
\font\tenmsb=msbm10 \font\sevenmsb=msbm7 \font\fivemsb=msbm5

\catcode`\@=11 \ifx\amstexloaded@\relax\catcode`\@=\active
\endinput\else\let\amstexloaded@\relax\fi
\def\spaces@{\space\space\space\space\space}
\def\spaces@@{\spaces@\spaces@\spaces@\spaces@\spaces@}
\def\space@.{\futurelet\space@\relax}
\space@.
\def\Err@#1{\errhelp\defaulthelp@\errmessage{AmS-TeX error: #1}}
\def\relaxnext@{\let\next\relax}
\def\accentfam@{7}
\def\noaccents@{\def\accentfam@{0}}
\def\Cal{\relaxnext@\ifmmode\let\next\Cal@\else
\def\next{\Err@{Use \string\Cal\space only in math mode}}\fi\next}
\def\Cal@#1{{\Cal@@{#1}}}
\def\Cal@@#1{\noaccents@\fam\tw@#1}
\def\Bbb{\relaxnext@\ifmmode\let\next\Bbb@\else
\def\next{\Err@{Use \string\Bbb\space only in math mode}}\fi\next}
\def\Bbb@#1{{\Bbb@@{#1}}}
\def\Bbb@@#1{\noaccents@\fam\msbfam#1}
\newfam\msbfam
\textfont\msbfam=\tenmsb \scriptfont\msbfam=\sevenmsb
\scriptscriptfont\msbfam=\fivemsb

\def\NN{{\mathbb{N}}}

\def\RR{{\mathbb{ R}}}

\def\bp{{\mathbf{p}}}
\def\bx{{\mathbf{x}}}
\def\bm{{\mathrm{M}}}
\def\nn{{ \nonumber}}

\def\beq{\begin{equation}}
\def\eeq{\end{equation}}

\def\qedbox{$\rlap{$\sqcap$}\sqcup$}

\newtheorem{Theorem}{Theorem}[section]
\newtheorem{Lemma}{Lemma}[section]

\newtheorem{Proposition}{Proposition}[section]
\newtheorem{Corollary}{Corollary}[section]
\newtheorem{Remark}{Remark}[section]

\newtheorem{Definition}{Definition}[section]

\@addtoreset{equation}{section}

\begin{document}

\title{ Intermittent behaviors in weakly coupled map lattices \footnote{This work is supported by NSFC of China (Grants 11031003, 11271183, 11471074), Simons Foundation.}
\author{
    Tiexiang Li\footnote{Department of Mathematics, Southeast Univeristy, Nanjing, China, txli@seu.edu.cn},\ \  Wen-wei Lin\footnote{Department of Mathematics, National Chiao Tung University, Hsinchu, 300, Taiwan, wwlin@math.nctu.edu.tw},\ \  Yiqian Wang\footnote{The corresponding author.}\ \footnote{Department of Mathematics,
Nanjing University, Nanjing, China, yiqianw@nju.edu.cn}  \ \ and
  \ Shing-Tung Yau\footnote{Department of Mathematics, Harvard University, Boston, United states, yau@math.harvard.edu
}
  }\date{}}\maketitle

  \begin{abstract}
  In this paper, we study intermittent behaviors of coupled piecewise-expanding map lattices with two nodes and a weak coupling. We show that the successive phase transition between ordered and disordered phases occurs for almost every orbit.  That is, we prove $\liminf_{n\rightarrow \infty}| x_1(n)-x_2(n)|=0$ and  $\limsup_{n\rightarrow \infty}| x_1(n)-x_2(n)|\ge c_0>0$, where $x_1(n), x_2(n)$ correspond to the coordinates of two nodes at the iterative step $n$. We also prove the same conclusion for weakly coupled tent-map lattices with any multi-nodes.
\end{abstract}

\vskip 0.3cm
\section{Introduction}

In this paper, we study the intermittent dynamical behavior of weakly-coupled piecewise-expanding map lattices. Let $f: [0, 1]\rightarrow [0, 1]$ be a piecewise expanding map,
%Denote $\bp=(x_1,\cdots, x_m)^\top\in\RR^m$, where $m$ is the number of the nodes in the lattice.
 $I$ be the $m\times m$ identity matrix and
$A$ be an $m\times m$ symmetric matrix satisfying $A\mathbf{e}=0$,  where $\mathbf{e}= [1,\cdots, 1]^\top$. Consider the dynamical system defined by a coupled map lattice:
\beq\label{multi-node}
T:\quad \bx(n+1)=(I+cA)  \mathbf{f}(\bx(n)),
\eeq
where $c$ is the coupling coefficient, $\bx(n)=[x_1(n),\cdots, x_m(n)]^\top\in [0, 1]^m$ for  $n\in\NN\cup \{0\}$ and $\mathbf{f}(\bx(n))=[f(x_1(n)),\cdots, f(x_m(n))]^\top$.   In case of no confusion, we also use bold letters $\bx$ or $\bp=(x_1,\cdots,x_m)$ to denote points in $[0,1]^m$.

Because of $A\mathbf{e}=0$, it can be easily seen that the diagonal $D_{\mathrm{syn}}=\{(x_1,\cdots, x_m)\in [0, 1]^m\mid x_1=\cdots=x_m\}$ is an invariant set for synchronized points of $T$. An interesting question on the dynamical behavior of the coupled map lattice (\ref{multi-node}) can be raised as whether $D_{\mathrm{syn}}$ is a global attractor, or equivalently, whether synchronization occurs for (\ref{multi-node}). There have been plenty of results on the study of synchronization when $f$ generates a chaotic dynamical system. Common examples include the tent maps and the Logistic maps, one can see \cite{Be-Ca,Li-Yorke,Young} and references therein. It has been shown in these results that chaotic synchronization can occur only if $c$ is far from zero.
That is, chaotic synchronization can not occur for small coupling strength.

However, a more complicated phenomenon has been found by numerical simulations when $c$ is small, i.e., when the coupling strength is weak. Roughly speaking, it is found that a typical orbit can enter into and exits slowly from an arbitrarily small neighborhood of $D_{\rm syn}$ for infinite times. In other word, the successive phase transition between being close to the diagonal and being far from the diagonal can happen. We call this phenomenon  as pseudo-synchronization.

The pseudo-synchronization is closely related to the clustering phenomenon in global coupled map lattices by  Kaneko et al. \cite{FKa,Ka1,Ka2,Ka3,Ka4,Ka5,KaY}. In numerical experiments, it showed that when (\ref{multi-node}) is a globally coupling system with large $m$, elements differentiate into some clusters, and  elements in each cluster oscillate synchronously, while the behaviors in different clusters are various. Moreover, the differentiation by clustering is a temporal  behavior in nature \cite{KaY}.  One can easily see that the pseudo-synchronization is a special case of the temporal clustering. In fact,  the temporal clustering is also found in all systems of (\ref{multi-node}) with small $c$. Similar behaviors were also widely explored  in weakly coupled continuous-time chaotic systems. For example, the successive phase transition between bursting and spiking was discovered in the study of epilepsy, see \cite{GNHL,GHZ} and references therein. More related results are shown in \cite{BKOVZ,HI,Pe,PCJFM} and references therein. To provide a mathematical proof for the mechanism of pseudo-synchronization for weakly-coupled map lattices is one of motivations of this paper.

On the other hand, there are a series of mathematical results on dynamical behaviors of weakly-coupled map lattices. In \cite{Keller1},  Keller showed that the existence of unique absolutely continuous invariant measure for weakly-coupled tent maps. Keller and Liverani \cite{Keller-Liverani0} proved the existence of the unique SRB measure for a wide range of multi-dimensional weakly coupled map lattices. They also showed the exponential decay of correlations in time and space in some one-dimensional lattices of weakly coupled piecewise expanding
interval maps \cite {Keller-Liverani1}.
More further results can be found in \cite{Bardet-Keller,Keller-Kunzle,Liverani1} and references therein. Another motivation of this paper is thus  to provide more informations on dynamical behaviors of weakly-coupled map lattices.
%partially improve the result of Keller by proving that absolutely continuous invariant measure $m$ found by him is in fact equivalent to Lebesgue measure $Leb$ in the sense that there exists two constant $0<c_1<c_2$ such that for any measurable set $S$, it holds that $c_2\cdot Leb(S)\le m(S)\le C_2\cdot Leb(S).$

In this paper, we will prove the occurrence of successive phase transitions for almost every point in the sense of Lebesgue measure
for the following weakly-coupled map lattices, where $f$ is the tent map or its perturbation.

In (\ref{multi-node}), when $m=2$,  we have
%\begin{equation}\label{two-node-model}
%T:\quad \left\{\begin{array}{ll}
%x(n+1)&=(1-c)f(x(n))+cf(x_1(n))\\
%y(n+1)&=cf(x(n))+(1-c)f(x_1(n)),
%\end{array}
%\right.
%\end{equation}
%In the study of chaotic synchronization, a well-known model is
the following coupled map lattice:
\begin{equation}\label{two-node-model}
T: \begin{cases}
&x_1(n+1)=(1-c)f(x_1(n))+cf(x_2(n))\\
&x_2(n+1)=cf(x_1(n))+(1-c)f(x_2(n))
\end{cases}.
\end{equation}
Let $\mathrm{dist} (A, B)$ denote the distance between two points/sets $A$ and $B$.
%Similarly, we use $dist(A, S)$ to denote the distance between a point $A$ and a set $S$.
Numerical simulation shows that when $f(x)$ is piecewise-expanding and close to the standard tent map, and $c$ is smaller than some $c_+>0$, the pseudo-synchronization occurs for the system (\ref{two-node-model}). That is,
\begin{equation}\label{order-part}
\liminf_{n\rightarrow \infty} \mathrm{dist} (\bx(n), D_{\mathrm{syn}})=0,
\end{equation}
and
\begin{equation}\label{disorder-part}
\limsup_{n\rightarrow \infty} \mathrm{dist} (\bx(n), D_{\mathrm{syn}})\ge \gamma_0>0,
\end{equation}
for some $\gamma_0>0$. In other word, the successive phase transition between the order phase (close to $D_{\mathrm{syn}}$) and the disorder phase (far from $D_{\mathrm{syn}}$) occurs for almost every orbit. Similar behaviors are also found for the multi-node cases.

Obviously, (\ref{order-part}) and (\ref{disorder-part}) are equivalent to the following equations, respectively,
$$
\liminf_{n\rightarrow \infty} | x_1(n)-x_2(n) | =0,
$$
and
$$
\limsup_{n\rightarrow \infty}  | x_1(n)-x_2(n) | \ge \gamma_0>0.
$$

In this  paper, we will provide such a series of mathematical proofs for Theorem \ref{main-theorem}-Theorem \ref{multi-node-linear} as follows.
\begin{Theorem}\label{main-theorem}
Consider the system (\ref{two-node-model}) with $f(x)=1-2 | x-1/2 | ,\ x\in [0, 1]$ being the standard tent map. There exists $0<c_+\le 1/4$ such that if the coupling coefficient $0\leq c<c_+$, then there exists a constant $\gamma_0>0$ such that for almost every initial point $(x_1(0),x_2(0))\in [0, 1]^2$, (\ref{order-part}) and (\ref{disorder-part})  hold true. \end{Theorem}
\begin{Remark}
With a more careful estimate, the same conclusion as in Theorem \ref{main-theorem} can be proved for all $0\le c<1/4$ {\rm(}or $3/4<c\le 1${\rm)}. On the other hand, it is not difficult to prove that for each $1/4\le c\le 3/4$,  the synchronization occurs. Thus we conclude that $c=1/4$ is the bifurcation point between synchronization and successive phase transition.
\end{Remark}
\begin{Remark}\label{piecewise-linear}
The function $f(x)$ in the system (\ref{two-node-model}) need not be the standard tent map. We can prove that there exist constants $c_+, \alpha_+>0$ such that the same conclusion as in Theorem \ref{main-theorem} holds true
for the general tent map
\begin{equation}\label{gen-tent-map}
 f(x)=\begin{cases}
&(2-\alpha_1)x,\qquad 0<x\le 1/2-\alpha_2\\
&\frac{1-2\alpha_2}{1+2\alpha_2}  (2-\alpha_1)(1-x),\quad 1/2-\alpha_2<x\le 1
\end{cases}
\end{equation}
with $ | c | \le c_+$ and $ | \alpha_i | <\alpha_+,\ i=1, 2$. Moreover, the result can be extended to the general piecewise linear continuous function $f$ with slopes being large enough.

%In fact, the phase space was divided into eight parts by the lines $x=1/2-\gamma_2$, $y=1/2-\gamma_2$, $y=x$ and the segments $AO$ and $BO$, where $A=(0,1),\ B=(1,0),\ O=(1/2-\gamma_2, 1/2-\gamma_2)$. Then for each sufficiently short segment $l$ in one of the eight parts, there are at most two folds among $(T^{k+1}l, \cdots, T^{k+6}l)$ for any $k$. Moreover, $ | Tl| \ge \max\{\frac{1-2\gamma_2}{1+2\gamma_2}, 1\}\cdot (2-\gamma_1)(1-c)\cdot |l|$, where $|\cdot|$ denotes the length.  Then the conclusion follows from the same argument in the proof of the main theorem.
\end{Remark}

The most important contribution of this paper is that the conclusion in Theorem \ref{main-theorem} can be generalized to the case that $f$ is piecewise expanding. More precisely, we will prove that
\begin{Theorem}\label{nonlinear}
Let $f_0(x)=1-s|x-1/2|, x\in [0,1]$, where $s=2-s_0$ with $0\le s_0<1$.
 and $g(x)$ be a ${\cal C}^2$-smooth function on $x\in [0,1]$ such that   $f(x)=f_0(x)+g(x)\in [0,1]$ for each $x$.  Then there exist three small constants $c_+, s_{0+}, \eta>0$
such that if $0\le c< c_+$, $s_0\le s_{0+}$ and $\|g\|_{{\cal C}^2}<\eta$, there exists a constant $\gamma_0>0$ such that for almost every initial point $(x_1(0),x_2(0))\in [0,1]^2$, (\ref{order-part}) and (\ref{disorder-part})  hold true for the system (\ref{two-node-model}).
\end{Theorem}
Theorem \ref{main-theorem} can also be extended to the  multi-node case. \vskip -0.3cm
\vskip -0.3cm\begin{Theorem}\label{multi-node-linear}
Let $f$ be the standard tent map and consider the coupled tent map lattices (\ref{multi-node}) with $A$ being $m\times m$ symmetric matrix satisfying $A\mathbf{e}=0$. There exists $c_+>0$ such that if the coupling coefficient $0\leq c<c_+$, then there exists a constant $\gamma_0>0$, such that for almost every initial point $(x_1(0),\cdots,x_m(0))\in [0,1]^m$, (\ref{order-part}) and (\ref{disorder-part}) hold true.
\end{Theorem}
\begin{Remark}\label{nonpiecewise-linear}
Remark \ref{piecewise-linear} is also applicable for Theorem \ref{nonlinear} and \ref{multi-node-linear}.
%\begin{figure}[H]
%\centering
%\includegraphics[height = 7cm, width = 9 cm]{report_pic_9.jpg}
%%\caption{ÍŒÆ¬ÏÂÃæµÄÎÄ×ÖËµÃ÷.}
%%\label{fig:XXXX} % ÓÃÓÚœ»²æÒýÓÃ
%\end{figure}
\end{Remark}
 Theorem \ref{main-theorem} is a special case of Theorem \ref{nonlinear} or \ref{multi-node-linear} without regard to the fact we can obtain a larger $c_+$ in Theorem \ref{main-theorem} than in other two theorems. However, we will still give the proof of Theorem \ref{main-theorem} first, since it is helpful for readers to understand the key idea of the proof as well as more complicated cases considered in Theorem \ref{nonlinear}.
 \begin{Remark} Theorem \ref{multi-node-linear} can also be obtained from Propositions 5 and 7 in \cite{Keller1}. We present the proof of it here since potentially we can combine the method in the proofs of Theorem \ref{nonlinear} and  \ref{multi-node-linear} to prove the intermittent behaviors for coupled piecewise-expanding map lattices with any multi-nodes. In contract,  the propositions 5 and 7 in \cite{Keller1} seem to work for tent maps only and is difficult to be applied to general piecewise-expanding maps.
 \end{Remark}
\vskip 0.3cm
The remaining  part of this paper is organized as follows. In Section 2, we give a key iteration lemma as the base for the main proof. In Section 3 and 4, we prove (\ref{order-part}) (the ordered part) and (\ref{disorder-part}) (the disordered part) of Theorem \ref{main-theorem}. The proof of Theorem \ref{nonlinear} is given in Section 5. In the last section, we will prove Theorem \ref{multi-node-linear}.
\section{The basic idea and the key lemma}
In this section, we will describe our intuition for the proof. From the observation, we then provide a key iteration lemma, which is the base for the main proof.

 Roughly speaking, for any set $S$ with a small measure in some sense, from the local expansion of the map $T$, we observe that  the measure of $T^j(S)$ will become large enough for some large $j$ such that
 $T^j(S)\cap D_{\mathrm{syn}}\not=\emptyset$. If $T^i(S)$ also satisfies some `good' property for $i=0, 1,\cdots, j$ (say, $T^i(S)$ is a segment or convex region), then we can show that  for any neighborhood of the diagonal $D_{\mathrm{syn}}$ there is a constant $m_0>0$ such that there exists a subset $S_0$ of $S$ satisfying that {\rm (i)} for each point $\bp \in S_0$, $T^{j}(\bp )$ is in the neighborhood of $D_{\mathrm{syn}}$; {\rm (ii)} $\bm(S_0)\ge m_0  \bm(S)$, where $\bm(\cdot)$ denotes the Lebesgue measure or a domain of the length of a simple curve. Then {\rm (i)} and {\rm (ii)} will imply  (\ref{order-part}) holds true for a set of full measure.

% The (local) expansion of $T$ is the key to obtain (i) and {\rm (ii)}. Intuitively, the images of the set $S$ will become larger and %larger under the iteration of $T$. When the image is large enough, it is inevitable that the intersection of the image and %$\{x_1=x_2=\cdots=x_m\}$ is nonempty. and the `good' property of $S$ will ensure {\rm (ii)}.

 However, since $T$ is not one to one and thus is not globally expanding, the actual picture is much more complicated than the one described above.
In fact, since $T$ is not one-to-one, usually the `good' property of a set $S$ is not preserved by its image $T(S)$. Without this property, it is impossible to obtain {\rm (ii)}.
On the other hand, let $[0, 1]^m=\cup D_J$, where $D_J$ are $2^m$ small hypercubes of $[0, 1]^m$ divided by the planes $x_i=1/2,\ 1\le i\le m$.  From the definition of $f$, one can see that $T: D_J\rightarrow [0, 1]^m$ is one to one for each $J$. Moreover,
$T(D_J\cap S)$ will keep the `good' property of
$S$. For this reason, we have to divide $S$ into $S\cap D_J$, and consider the iterations of $T$ on each of them individually.

For convenience, we say $S$ has $i_1$ components if there are $i_1$ nonempty sets among all $S\cap D_J$. Furthermore, consider a set $D\subset [0, 1]^m$. For any set $S$ with components $S_{1,1}, \cdots, S_{1, i_1}$, if there are exactly $\hat{i}_1$ components $S_{1, 1_j},  \ j=1,\cdots,\hat{i}_1\leq i_1$ among them such that $S_{1, 1_j}\cap D\not=\emptyset$ for each $1\le j\le \hat{i}_1$, we say $S$ has $\hat{i}_1$ components in $D$. For the set $S$ stated above, suppose for each $1\le j\le i_1$, $T(S_{1,j})$ has $k_1(j)$ components, we say $T(S)$ has $\sum_{j=1}^{i_1} k_1(j)$ components. Similarly, suppose for each $1\le j\le \hat{i}_1$, $T(S_{1,1_j})$ has $\hat{k}_1(j)$ components in $D$, we say $T(S)$ has $\sum_{j=1}^{\hat{i}_1}\hat{k}_1(j)$ components in $D$. Inductively, suppose $T^{l}(S)$ has components $S_{l,1},\cdots,S_{l,i_l}$ and has components  $S_{l,l_j}, \ j=1,\cdots, \hat{i}_l\leq i_{l}$ in $D$. Furthermore, for each $1\leq j\leq i_l$, $T(S_{l,j})$ has $k_l(j)$ components, and for $1\leq j\leq \hat{i}_l$, $T(S_{l,l_j})$ has $\hat{k}_l(j)$ components in $D$. We say that $T^{l+1}(S)$ has $\sum_{j=1}^{i_{l}}k_l(j)$ components and $\sum_{j=1}^{\hat{i}_{l}}\hat{k}_l(j)$ components in $D$, respectively. In the same way, we can give the definitions for $T^{-l}(S)$ (in $D$). Note that each component $\Omega$ of $T^{l}(S)$ corresponds a subset $\Omega_0\subset S$ such that $T^{l}:\ \Omega_0\rightarrow \Omega$ is a homeomorphism.

Obviously, the measure of each component of $S$ is usually strictly less than the measure of $S$. Moreover, images of each component may have more than one component. Thus we need to show that averagely the local expansion of the map will surpass the dividing action by $x_i=1/2$, $i=1, \cdots, m$ on a set so that the measures of the components will keep increasing as the iteration goes forward only if the intersection of the corresponding set and $D_{\mathrm{syn}}$ is empty. More detailed, we have the key iteration lemma and its corollary as below.
\vskip 0.2cm
%Define $G_{\epsilon}=\{(x_1,
  %  \cdots, x_m)\in [0,1]^m| \sum_{1\le i\le j\le m}(x_i-x_j)^2\le \epsilon^2\}$.
%we say a point $P\in D$ is $\epsilon$-good if there exists $l(P)\in \NN$ such that $T^{l(P)}P\in G_{\epsilon}$. Similarly, we %say a set $S\subset D$ is $\epsilon$-good if so is every point in it. Furthermore, we say
% set $S\subset D$ is $(c_0,\epsilon)$-good if there exists a set $\bar{S}\subset S$ with $\bm(\bar{S})\ge c_0\cdot \bm(S)$ such %that
 %$\bar{S}$ is $\epsilon$-good.

Consider the coupled map lattice $T$ in \eqref{multi-node} with $f(x)$ differential for $x\not=1/2$. Let $S\subset  [0,1]^m$ be a simple curve or measurable set. We define
 $$E_{+}(c)=\begin{cases}
\sup_{\bp \in  [0,1]^m} | \det(JT(\bp ,c)) | = | \det(A(c)) | \nu_+^m  &  {\rm\ with \ respect\ to\  a\ measurable\ set}\ S\\
\sup_{\bp \in  [0,1]^m}\|JT(\bp ,c)\|  & {\rm\ with \ respect\ to\  a\ simple\ curve}\ S\end{cases},$$ where $A(c)\equiv I+cA$ is the coupling matrix, $JT(\bp,\cdot )$ be the Jacobian matrix of $T$ at $\bp$ (if it exists) and $\nu_+=\sup_{x\not=1/2} | f'(x) | $. Similarly, we define
$$E_{-}(c)=\begin{cases}
\inf_{\bp \in  [0,1]^m} | \det(JT(\bp ,c)) | = | \det(A(c)) |  \nu_-^m& {\rm\ with \ respect\ to\  a\ measurable\ set}\ S\\
\inf_{\bp \in  [0,1]^m}\|JT(\bp ,c)\|&\!\! {\rm\ with \ respect\ to\  a\ simple\ curve}\ S\end{cases},$$ where $\nu_-=\inf_{x\not=1/2} | f'(x) | $.
Obviously, for any simple curve or measurable set $S$ in some small cube $D_J$ of the phase space, we have
 $E_-(c)  \bm(S)\le \bm(T(S))\le E_+(c)  \bm(S)$.

 For example, if $f$ is the tent map and $m=2$, then $E_{\pm}(c)$ equals $4(1-2c)$ for a measurable set and $E_+(c)=2$, $E_-(c)=2(1-2c)$ for a curve.

%\noindent Proposition \ref{firststep-L2.2} can be obtained by applying the following key lemma inductively.
For any real number $x$,  we define $\lfloor x\rfloor=\max\{i \ \mathrm{ is \ an  \ integer}  | i\le x\}$. Let $D$ be a domain in the phase space $[0,1]^m$.
\begin{Definition}
Let $\delta>0$, $m_0,\ a\in \NN$. We say a simple curve or measurable set $S$ in $D$ is $(\delta, m_0, a)$-good if $S$ lies in some small cube of $[0,1]^m$ with $\bm(S)\le \delta \bm(D)$ and for each $i\ge 0$ and each component $S_1$ of $T^i(S)$ with $\bm(S_1)\le \delta \bm(D)$, it holds that $T^{m_0}(S_1)$ has at most $a$ components.
\end{Definition}

\begin{Lemma}\label{iteration-L} {\rm (Iteration Lemma)}
Assume $E_+(c)\ge E_-(c)>1$. Let $D$ be a domain in the phase space $[0, 1]^m$.
%Suppose that there exist constants $m_0,\ a\in \NN$ and $0<\delta_1<1$ satisfying  $a<E_-^{m_0}(c)$ such that for any curve or domain $S\subset D$ satisfying $\bm(S)\le \delta_1  \bm(D)$, there exist disjoint sets $S_{j}\subset S, j=1, 2,\cdots, a$ such that $T^{m_0}$ is $1\!\!-\!\!1$ on $S_j$ for each $j\ge 1$.
Suppose $0<\delta_1<1$ and $a<E_-^{m_0}(c)$ with $a, m_0\in \NN$.

 Let  $1<\mu<\left(1-\log_{E_+(c)}\frac{E_-(c)}{a^{1/m_0}}\right)^{-1}$, $d=1-\left(1-\log_{E_+(c)}\frac{E_-(c)}{a^{1/m_0}}\right)\mu>0$ and
$F(c)=a (\frac{E_-(c)}{a^{1/m_0}})^{1-\log_{E_+(c)}\delta_1}$. Define $N_0=\lfloor\log_{\mu}(d^{-1}\log_2F(c))\rfloor$.
 Suppose $N> N_0$ and for any $(\delta_1, m_0, a)$-good curve or measurable set $\Omega\subset D$ with
$\bm(\Omega)\in [2^{-\mu^{N+1}}\bm(D), 2^{-\mu^N}\bm(D)]$,
define $k(N)=\left\lfloor-\log_{E_+(c)} \frac{\bm(\Omega)}{\delta_1 \bm(D)}\right\rfloor$.
Then there exist some disjoint subcurves or measurable subsets $\Omega_{j}\subset \Omega,\ j=1, \cdots $ such that {\rm(i)} for each $j\ge 1$ it holds that $T^{k(N)}(\Omega_j)$ is a component of $T^{k(N)}(\Omega)$ with $\bm(T^{k(N)}(\Omega_j))\ge 2^{-\mu^N}\bm(D)$; {\rm(ii)} $\bm(\cup_{j\ge 1}\Omega_j)\ge (1-F(c)  2^{-d \mu^{N}}) \bm(\Omega)$.
\end{Lemma}

\begin{proof}
  We iterate the map on $\Omega$ for $k(N)$ times. Then there are at most $a^{\left\lfloor\frac{k(N)}{m_0}\right\rfloor+1}$ disjoint set $\widehat{\Omega}_{j}\subset \Omega$ such that $T^{{k(N)}}(\widehat{\Omega}_{j})$ is a component of $T^{{k(N)}}({\Omega})$ for each $j$ (note that the assumption that $\bm(T^{i}(\widehat{\Omega}_j))\le \delta_1  \bm(D)$ is valid for each $\widehat{\Omega}_{j}$
  and each $i\le k(N)$).

  Then the total measure of all $\widehat{\Omega}_j$'s satisfying $\bm(T^{k(N)}(\widehat{\Omega}_j))\le 2^{-\mu^N}\bm(D)$ is less than
   \begin{align*}
&2^{-\mu^N}\bm(D)  a^{\left\lfloor\frac{k(N)}{m_0}\right\rfloor+1}  E_-(c)^{-k(N)} \qquad \quad{\rm (since}\ \bm(\widehat{\Omega}_j)\le E_-(c)^{-k(N)}  \bm(T^{k(N)}(\widehat{\Omega}_j)))\\
\le & a  2^{-\mu^N}\bm(D)  \left(\frac{E_-(c)}{a^{1/m_0}}\right)^{-k(N)}
\\
\le & a  2^{-\mu^N}\bm(D)  \left(\frac{E_-(c)}{a^{1/m_0}}\right)^{\left\lfloor\log_{E_+(c)} \frac{\bm(\Omega)}{\delta_1  \bm(D)}\right\rfloor+1}\\
\le & a  2^{-\mu^N}\bm(D)  \left(\frac{E_-(c)}{a^{1/m_0}}\right)^{1-\log_{E_+(c)} \delta_1}\left(\frac{E_-(c)}{a^{1/m_0}}\right)^{-\log_{E_+(c)} \bm(D)}\left(\frac{E_-(c)}{a^{1/m_0}}\right)^{\log_{E_+(c)} \bm(\Omega)}\\
= & F(c) 2^{-\mu^N}\bm(D)^{1-\log_{E_+(c)} \left(\frac{E_-(c)}{a^{1/m_0}}\right)}\bm(\Omega)^{\log_{E_+(c)} \left(\frac{E_-(c)}{a^{1/m_0}}\right)}\qquad {\rm (since}\  a^{\log_b c}=c^{\log_b a}).
% \\ \le& F(c)  2^{-\mu^N}  \bm(\Omega)^{\log_{E_+(c)} \left(\frac{E_-(c)}{a^{1/m_0}}\right)}.\qquad \qquad{\rm (since }\ \bm(D)<1\ {\rm and } \ 1<\frac{E_-(c)}{a^{1/m_0}}<E_+(c) )
\end{align*}
Hence it possesses a portion of $\Omega$ less than
\begin{align*}
&F(c)  2^{-\mu^N} \left (\frac{\bm(\Omega)}{\bm(D)}\right)^{\log_{E_+(c)} \left(\frac{E_-(c)}{a^{1/m_0}}\right)-1}\\
\le & F(c)  2^{-\mu^N}  \left (2^{-\mu^{N+1}}\right)^{\log_{E_+(c)} \left(\frac{E_-(c)}{a^{1/m_0}}\right)-1}
\le  F(c)  2^{-d \mu^N}.
\end{align*}%where $d=1-\left(1-\log_{E_+(c)}\frac{E_-(c)}{a^{1/m_0}}\right) \mu>0$ since %$\mu<\left(1-\log_{E_+(c)}\frac{E_-(c)}{a^{1/m_0}}\right)^{-1}$.
Choose $\Omega_1, \Omega_2
\cdots,$ be all $\widehat{\Omega}_j$ with a measure larger than $2^{-\mu^N}\bm(D)$ and the proof is completed.\hfill\qedbox
\end{proof}

\vskip 0.3cm
%\begin{Lemma}\label{iteration-L} {\rm (The Iteration Lemma)} Let $N>3??$ and $S_{N+1}$ is a horizontal or vertical segment  with a length between $[2^{-2^{N+1}}, 2^{-2^N}]$. Then there exists a collection of subsegments possessing a proportion larger than $1-2^{-d*2^{N}}$  such that for any segment $\Gamma$ among them, $T^{k(N)}\Gamma$ is a segment longer than $2^{-2^N}$, where $d=1/3+2\log_2(1-2c)>0$ and $k(N)=[-\log_2 L_N-2^3]$ with $L_N=|S_{N+1}|$.
%\end{Lemma}
%\begin{Remark}\label{RK2.3}
%If necessary, the scale of the length for $S_{N+1}$ can be replaced by  $[2^{-a^{N+1}}, 2^{-a^N}]$ with $a=2^{1/n}$ where $n\in \NN$. Thus the condition on the parameter $c$ can be weakened.
%\end{Remark}
\begin{Corollary}\label{Corollary}
Let the domain $D$, $m_0, a, \mu, N_0$ and $d$ be as in Lemma \ref{iteration-L}. Then there exists a constant $c_1>0$ such that for any $(\delta_1, m_0, a)$-good curve or measurable set
 $\Omega\subset D$ with $\bm(\Omega)\le \delta_1 \bm(D)$, there exist disjoint subcurves or measurable subsets $\Omega_i\subset \Omega$, $i=1, 2,\cdots,$ such that {\rm(i)} for any $i$, there exists $k(i)$ such that $T^{k(i)}(\Omega_i)$ is a component of $T^{k(i)}(\Omega)$ and $\bm(T^{k(i)}(\Omega_i))\ge 2^{-\mu^{N_0}}  \bm(D)$; {\rm(ii)} $\bm(\cup_i \Omega_i)\ge c_1  \bm(\Omega).$
\end{Corollary}
\vskip 0.3cm
\begin{proof} Let $N$ be the unique integer such that $2^{-\mu^{N+1}}\bm(D)\le \bm(\Omega)<2^{-\mu^N}\bm(D)$ and denote $\Omega_{N+1}=\Omega$.
 Applying Iteration Lemma \ref{iteration-L} on $\Omega_{N+1}$, there exist disjoint $\Omega_{N+1}^i\subset \Omega_{N+1},\ i=1, 2, \cdots,$ such that  ${\rm(a)}_{N+1}$\ \
 $T^{k(N)} (\Omega_{N+1}^i)$ is a component of $T^{k(N)} (\Omega_{N+1})$ and $ \bm(T^{k(N)} (\Omega_{N+1}^i)) $ $ \ge 2^{-\mu^N} \bm(D)$ for each $i$; ${\rm(b)}_{N+1}$\ \  $\bm(\cup_{i\ge 1}\Omega_{N+1}^i)\ge (1-F(c)  2^{-d\mu^{N}}) \bm(\Omega_{N+1})$.

Let $I_{N}=\{i\mid \bm(T^{k(N)}(\Omega_{N+1}^i))\in [2^{-\mu^N}\bm(D), 2^{-\mu^{N-1}}\bm(D))\}$ and denote the set of all other $i$ by $I'_N$. Let $i\in I_N$ and applying Iteration Lemma \ref{iteration-L} on $T^{k(N)}(\Omega_{N+1}^i)$ and we have that there exist disjoint subcurves or measurable subsets $\Omega_{N}^{i,j}\subset \Omega_{N+1}^i,\ j=1, 2, \cdots$ and $k_i(N)$ such that ${\rm(a)}_{N}^i$
 $T^{k_i(N)}(\Omega_{N}^{i,j})$ is a component of $T^{k_i(N)}(\Omega_{N+1}^{i})$  and $\bm(T^{k_i(N)}(\Omega_{N}^{i,j}))\ge 2^{-\mu^{N-1}}\bm(D)$ for each $j$; ${\rm(b)}_{N}^i$ $\bm(\cup_{j\ge 1}\Omega_{N}^{i,j})\ge (1-F(c)  2^{-d\mu^{N-1}}) \bm(\Omega_{N+1}^{i})$. Thus we have
 $$
 \bm\left(\cup_{i\in I_N}(\cup_j \Omega_{N}^{i,j})\cup(\cup_{i\in I'_N}\Omega_{N+1}^i)\right)\ge (1-F(c)  2^{-d\mu^{N}})(1-F(c)  2^{-d\mu^{N-1}}) \bm(\Omega_{N+1}).
 $$
Moreover, the sets on the left hand side in the above inequality are disjoint with each other, and for each set $\widetilde{\Omega}$ of them there exists a $k(\widetilde{\Omega})$ such that $T^{k(\widetilde{\Omega})}$ is a component and $\bm(T^{k(\widetilde{\Omega})}(\widetilde{\Omega}))\ge 2^{-\mu^{N-1}}\bm(D)$.

 By induction, we can obtain the existence of $\Omega_i, i=1, 2, \cdots$ such that {\rm (i)} and {\rm (ii)} hold true, where $c_1=\prod_{j=N_0}^{N}(1-F(c)  (2^d)^{-\mu^{j}})$, which has a  positive lower bound for all $N$. In fact, it is sufficient to prove that
$\prod_{j=N_0}^{\infty}(1-F(c)(2^d)^{-\mu^{j}})>0$, which can be obtained from the fact that
$\sum_{j=N_0}^{\infty}\ln (1-F(c)(2^d)^{-\mu^{j}})\geq -F(c)\sum_{j=N_0}^\infty (2^d)^{-\mu^j}>-\infty$ (bounded below). This completes the proof.\qquad \hfill\qedbox
\end{proof}

\begin{Remark}
Corollary \ref{Corollary} holds true for any $1<\mu<\left(1-\log_{E_+(c)}\frac{E_-(c)}{a^{1/m_0}}\right)^{-1}$. As $\mu\rightarrow 1$, we obtain the upper bound $\log_{E_+(c)}(E_-(c)  a^{-1/m_0})  F(c)^{-1}$ for $2^{-\mu^{N_0}}$.
\end{Remark}
%%%%%%%%%%%%%%%%%%%%%%%%%%%%%%%%%%%%%%%%%%%%%%%%%%%%%%%%%%%%%%%%%%%%%%%%%%%%%%%%%%%%%%%%%%%%%%%%%%%%%

\section{The ordered part}
In this section, we will prove the ordered part of Theorem \ref{main-theorem}, that is, we will prove (\ref{order-part}) holds true for almost every initial point.  For this purpose, it is sufficient to prove the following theorem.
%\vskip 1.5cm

%\begin{Corollary}
%The pseudo-synchronization will occur for almost every point $(, x_2(0))$ in $D$.
%\end{Corollary}
\vskip 0.3cm

\begin{Theorem}\label{liminf}
For any given $\epsilon>0$ and almost  every initial point $(x_1(0), x_2(0))$ in $[0, 1]^2$, it holds that
${\inf}_{n\in\NN}  \mathrm{dist}((x_1(n),x_2(n)),D_{\mathrm{syn}}) <\epsilon$.
\end{Theorem}
\vskip 0.2cm
\begin{Remark} \ Theorem \ref{liminf} implies that for almost every point, its orbit will enter the $\epsilon$-neighborhood of the diagonal $x_1=x_2$ for at least one time.
We will use it to prove that the orbit of almost every point will enter into (or stay in) the $\epsilon$-neighborhood of the diagonal $x_1=x_2$ for infinitely many times, which is just (\ref{order-part}).
\end{Remark}
\vskip 0.2cm

 \noindent{\it Proof of (\ref{order-part}) from Theorem \ref{liminf}} \quad
 For any positive integer $i$, let $$D_i=\{(x_1^{(i)}(0), x_2^{(i)}(0))\in [0, 1]^2\mid  \mathrm{dist}( (x_1^{(i)}(n),x_2^{(i)}(n)), D_{\mathrm{syn}})  \ge \frac{1}{i}, \text{for all }\ n\in\NN\}.$$ Obviously, to obtain (\ref{order-part}),
  it is sufficient to prove that $\bm(\cup_{i\ge 1} D_i)=0$ (note that $D_{syn}$ is an invariant set). Setting $\epsilon=\frac{1}{i}$ in Theorem \ref{liminf}, we obtain that $\bm(D_i)=0$, which immediately implies $\bm(\cup_{i\ge 1} D_i)=0$.
% \qquad
 % Otherwise, there is a region $D_0$ with a positive measure such that each orbit starting from
%$D_0$  enters $\epsilon$-neighbor of the synchronization manifold for only finite times. Then there exists some subset %$\tilde{D}_0$ of $D_0$ and some number $n_0\in \NN$ such that for all $n>n_0$, the image of $\tilde{D}_0$ under $n$'s %iteration will stay out of the $\epsilon$-neighbor of the diagonal. Obviously, the measure of the image is positive, which is %contradict to Theorem \ref{liminf}.
\hfill\qedbox

Let $G_{\epsilon}=\{\bp=(x_1,x_2)\in [0, 1]^2 | \mathrm{dist} (\bp,D_{\mathrm{syn}}) \le \epsilon\}$ and ${B}_{\epsilon}=[0, 1]^2\backslash G_{\epsilon}$. %We say a point $(x_1(0),x_2(0))$ in $[0, 1]^2$ is $\epsilon$-good if there exists some $k$ such that $(x(k),y(k))\in %G_{\epsilon}$. We say a subset in $[0, 1]^2$ is $\epsilon$-good if any point in it is $\epsilon$-good.
Theorem \ref{liminf} is a corollary of the following lemma.
\begin{Lemma}\label{L2.2}
There exists a constant $c_1>0$ such that for any $
0\le c<c_1$ and $\epsilon>0$, there exists $c_0\equiv c_0(c,\epsilon)>0$ such that any  segment $\Gamma$ with a slope $\pm 1$ in $[0, 1]^2$ has disjoint subsegments $\Gamma_1, \Gamma_2, \cdots$
 satisfying {\rm (i)} for any $i$, there exists $l_i$ such that
 $T^{l_i}(\Gamma_i)\subset G_{\epsilon}$; {\rm (ii)}
 $\bm(\cup_i\Gamma_i)\ge c_0  \bm(\Gamma)$.\end{Lemma}
\vskip 0.2cm
{\it Proof of Theorem \ref{liminf} from Lemma \ref{L2.2}}\ \qquad
From Lemma \ref{L2.2}, we have that for any segment $\Gamma$ with a slope $\pm 1$ in $[0, 1]^2$, there exist disjoint subsegments $\Gamma_1,\Gamma_2, \cdots$ satisfying {\rm (i)} and {\rm (ii)}.
 Obviously $\Gamma\backslash (\cup_i \Gamma_i)$
 is composed of a collection of disjoint subsegments of $\Gamma$, which is denoted by $\cup_j\Gamma_j'$.
 From {\rm (ii)}, we know that
 $\bm(\cup_j\Gamma_j')\le (1-c_0)\bm(\Gamma)$.

 Applying Lemma \ref{L2.2} again on each $\cup_j\Gamma_j'$,
 we obtain that for each $j$, there exist disjoint subsegments
 $\Gamma_{j,k}'\subset \Gamma_j',\ k=1, 2, \cdots$ satisfying
 $\bm(\cup_k\Gamma_{j,k}')\ge c_0  \bm(\Gamma_j')$
 and for any $j, k$, there exists $l(j,k)$ such that $T^{l(j,k)}
 (\Gamma_{j,k}')\subset G_{\epsilon}$. It is easy to see that
 $\cup_j(\Gamma_j'\backslash (\cup_k \Gamma_{j,k}'))$
is also composed of disjoint subsegments of $\Gamma$ and the total length of them is not larger than $(1-c_0)^2  \bm(\Gamma)$.
Inductively, we can obtain that for any $i$, we can find disjoint segments $\Gamma_{i,k}\subset \Gamma,\ k=1, 2, \cdots$ such that
for each $k$ there exists $l(i,k)$ such that
$T^{l(i,k)}(\Gamma_{i,k})\subset G_{\epsilon}$ and
$\bm(\Gamma\backslash \cup_{k}\Gamma_{i,k})\le (1-c_0)^i  \bm(\Gamma)$. Let $i\rightarrow \infty$, we obtain that the set of points in $\Gamma$ whose orbit is always out of $G_{\epsilon}$ is of measure zero. Then from Fubini's Theorem, we obtain Theorem \ref{liminf}.
\hfill \qedbox
\vskip 0.5cm

\noindent  The proof of Lemma \ref{L2.2} can be reduced to the following two propositions.
\begin{Proposition}\label{firststep-L2.2} There exists a constant $c_1>0$, such that for any $0\le c<c_1$, there is a  $\hat{c}_0\equiv \hat{c}_0(c)>0$, for any segment $\Gamma_0$ in one of four small squares of $[0, 1]^2$ with a slope $\pm 1$ and a length less than $\delta_1=2^{-16}$, there exists a collection of disjoint subsegments $\Gamma_i,\ i=1, 2, \cdots$ satisfying $\sum_i \bm(\Gamma_i)\ge\hat{c}_0 \bm(\Gamma_0)$ such that for each segment $\Gamma_i$,  there exists some $l(\Gamma_i)$ such that $T^{l(\Gamma_i)}(\Gamma_i)$ is a component of $T^{l(\Gamma_i)}(\Gamma_0)$ and $T^{l(\Gamma_i)}(\Gamma_i)$ is a segment with a slope  $\pm 1$ and a length larger than $\delta_1$.
\end{Proposition}

\noindent{\it Proof.}\quad \  From the condition on $\Gamma_0$
and the expansion of $T$, we have that $\bm(T(\Gamma_0))\ge
 2(1-2c)  \bm(\Gamma_0)$ From the definition of $T$, we know that for small $c$, number of components for short segments increases very slowly as the iterations go forward. For example, it can be easily seen that there are disjoint sets $\widehat{\Gamma}_i, i=1, 2, 3, 4$ such that $\cup_{i=1}^4 \widehat{\Gamma}_i=\Gamma_0$ and $T^6(\widehat{\Gamma}_i),\ i=1, 2, 3, 4$ are all the components of $T^6(\Gamma_0)$ (note that for small $c$, the image of $x_1, x_2=1/2$ under $T$ are close to $x_1=1$ and $x_2=1$, respectively. In addition, the images of $x_1, x_2=1$ under $T^i$ for $i\le 6$ is far from $x_1, x_2=1/2$).
 %Let $L_N$ be the length of $S_{N+1}$. We iterate the map on %$S_{N+1}$ for $k(N)=[-\log_2 L_N-2^3]$ times which is between %$[2^N-2^3, 2^{N+1}-2^3]$. Let $\lambda=2(1-2c)$  with $c$ the %coupling strength. From the defintion of $T$, we have that for %any curve $\Gamma$, it holds that $|T(\Gamma)|\le 2|\Gamma|$. %Then the total length of $T^{k(N)}S_{N+1}$ is smaller than %$\delta_1$. Note the fact that for a segment shorter than %$\delta_1$, among every 6 times iterations, there are at most 2 %times that the fold occurs. Thus there are totally at most
%$\frac{[-\log_2 L_N]}{3}$ times fold.
Applying  Iteration Lemma \ref{iteration-L} and Corollary \ref{Corollary} with $m_0=6, a=4, \mu=2, N_0=4$ and $D=[0,1]^2$, the conclusion is obtained.
\hfill \qedbox

\begin{Proposition} \label{step2-L2.2}
There exists a constant $c_1>0$ such that for any $0\le c<c_1$ and $\epsilon>0$, there exists  $c_2\equiv c_2(c,\epsilon)>0$ such that
for any segment $\Gamma_0$ in some small square of $[0, 1]^2$ with a slope  $\pm 1$ and a length larger than $\delta_1=2^{-16}$,  there  exists a segment $\Gamma\subset \Gamma_0$ and $l$ such that $T^l(\Gamma)\subset G_{\epsilon}$ and
$\bm(\Gamma)\ge c_2  \bm(\Gamma_0)$.
\end{Proposition}

{\it Proof.}\quad
 We first claim that there exists a constant $e>0$ such that for any $\Gamma_0$ in some small square with a slope $\pm 1$ and a length larger than $\delta_1=2^{-16}$, there exists a segment $\Gamma_1$ with a slope $\pm 1$ in the curve $T(\Gamma_0)$ or $T^2(\Gamma_0)$ such that $\Gamma_1$ is in some small squares and $\bm(\Gamma_1)\ge (1+e)\bm(\Gamma_0)$.

Since $\bm(T(\Gamma_0))\geq 2(1-2c)\bm(\Gamma_0)$, if $T(\Gamma_0)$ is in some small square, the claim is proved by setting $\Gamma_1=T(\Gamma_0)$.

 Thus we consider the case that the intersection between $T(\Gamma_0)$ and $x_2=1/2$ (or $x_1=1/2$) is nonempty. If both $x_1=1/2$ and $x_2=1/2$ have an intersection set with $T(\Gamma_0)$, then it is not difficult to see that $T(\Gamma_0)$ has an intersection set with $x_1+x_2=1$, this ends the proof of this proposition. Hence without loss of generality we assume $T(\Gamma_0)$ only crosses $x_2=1/2$ (or $x_1=1/2$).

   Let $\Gamma_{1,1}$ and $\Gamma_{1,2}$ be the components of $T(\Gamma_0)$ and $e=\frac{\lambda^2}{\lambda+1}-1$ with $\lambda=2(1-2c)$.
  It is obvious that $e>0$ for small $c$.

  From the expansibility of $T$, we have
   \beq \label{sum-length} \bm(\Gamma_{1,1})+\bm(\Gamma_{1,2})=\bm(T(\Gamma_0))\ge \lambda\bm(\Gamma_0).
   \eeq
   If $\bm(\Gamma_{1,1})\ge (1+e)\bm(\Gamma_0)$ or $\bm(\Gamma_{1,2})\ge (1+e)\bm(\Gamma_0)$, the claim is proved
 by choosing $\Gamma_1$  to be $\Gamma_{1,1}$ or $\Gamma_{1,2}$.

 Thus, in the following, we consider the case that both $\Gamma_{1,1}$ and $\Gamma_{1,2}$ are shorter than $(1+e)\bm(\Gamma_0)$.
 From (\ref{sum-length}) we have both $\Gamma_{1,1}$ and $\Gamma_{1,2}$ are longer than $(\lambda-(1+e))\bm(\Gamma_0)$.

 Let $T^2(\Gamma_0)=\Gamma_{2,1}\cup \Gamma_{2,2}$ with $\Gamma_{2,1}\cap \Gamma_{2,2}$ being a one-point set, where
  $\Gamma_{2,i}=T(\Gamma_{1,i})$ with a slope $(-1)^{i+1}$, $i=1, 2$ (see Figure \ref{fig:3.1}). By a direct computation, it holds that
 $$
 \bm(\Gamma_{2,i})\ge \lambda \bm(\Gamma_{1,i})\ge \lambda (\lambda-(1+e))\bm(\Gamma_0) = (1+e)\bm(\Gamma_0).$$
 Then if $\Gamma_{2,1}$ or $\Gamma_{2,2}$ is in some small square,
we complete the proof of the claim.

\begin{figure}[H]
\centering
\includegraphics[height = 7cm, width = 9 cm]{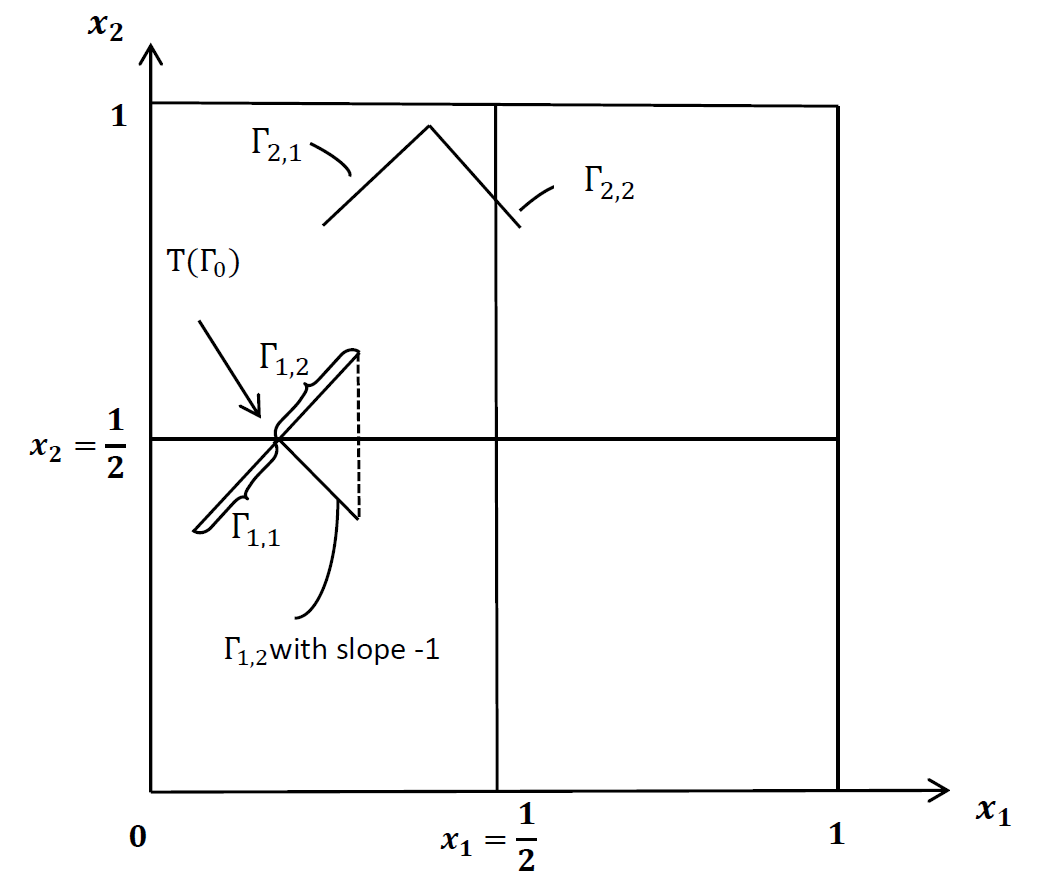}
\caption{Segments of $\Gamma_1=T(\Gamma_0)$ and $\Gamma_2=T^2(\Gamma_0)=\Gamma_{2,1}\cup \Gamma_{2,2}.$}
\label{fig:3.1} % ÓÃÓÚœ»²æÒýÓÃ
\end{figure}

 So we assume both $\Gamma_{2,1}$ and $\Gamma_{2,2}$ are not in some small square. In the following, we will prove it is impossible. Recall that the slopes of these two segments are $1$ and $-1$, respectively (see Figure \ref{fig:3.1}). Moreover, it is clear that
the lines where $\Gamma_{2,1}$ and $\Gamma_{2,2}$ lie in are symmetric with respect to a vertical line, which implies that it is impossible that
 both $\Gamma_{2,1}$ and $\Gamma_{2,2}$ have an intersection with $x_1=1/2$ (or $x_2=1/2$).
   This ends the proof of the claim.

 By induction, if $\Gamma_i$ has no intersection with $x_1=x_2$ for $i\ge 0$, we can find a segment $\Gamma_{i+1}\subset T(\Gamma_i)\ {\rm or\ } T^2(\Gamma_i)$ lying in some small square satisfies
 $\bm(\Gamma_{i+1})\ge (1+e)\bm(\Gamma_i)$.

 Since the set $[0, 1]^2$ has a finite diameter, there exists some $i_0\le \left\lfloor-\log_{1+e}\bm(\Gamma_0)\right\rfloor+1$ such that $\Gamma_{i_0}$ has an intersection with $x_1=x_2$ satisfying $\bm(\Gamma_{i_0}\cap G_{\epsilon})\ge \epsilon$. Since $\bm(T^{-1}(\Gamma))\ge 2^{-1}  \bm(\Gamma)$ for any curve $\Gamma$ in some small square and $\bm(\Gamma_0)\leq \frac{\sqrt{2}}{2}$, we have that
 $$
 \frac{\bm(T^{-i_0}(\Gamma_{i_0}\cap G_{\epsilon}))}{\bm(\Gamma_0)}\ge  \sqrt{2}   2^{-i_0} \epsilon \equiv c_2.
 $$
This completes the proof of the proposition.
\hfill\qedbox
%\begin{Remark}
%The same conclusion holds true for larger $c$ with a more precise estimate, since the frequency for the occurrence of fold is tends to 0. In fact, let $\delta_1\rightarrow 0$, the upper bound for $c$ tends to 1/4?.
%\end{Remark}
%
%
%\vskip 0.1cm
%
%
%\vskip 0.3cm

\begin{Remark}
With a smaller $\delta_1$, the same conclusion holds true for larger $c$, since the frequency for the occurrence of fold (i.e., the segment has nonempty intersection with $x_1=1/2$ or $x_2=1/2$) tends to 0 as $\delta_1\rightarrow 0$. In fact, let $\delta_1\rightarrow 0$, the upper bound for $c$ tends to 1/4.
\end{Remark}
\section{The disordered part}
In this section, we will prove {\rm(}\ref{disorder-part}{\rm)}, the disordered part of the main theorem, which states that for almost every point in the phase space, its orbit, although  will visit any neighborhood of $D_{\mathrm{syn}}$, will also be far away from $D_{\mathrm{syn}}$ for infinitely many times.
%\begin{Theorem}\label{asyncrhonization}
%There exists a (small) fixed number $c_0>0$ such that for almost every point $(x_0, y_0)$ in the phase space, we have
%$\limsup_{n\rightarrow \infty} |\bp_2(n)-y_n|\ge c_0.$
%\end{Theorem}\par
It is sufficient to prove that
\begin{Theorem}\label{equal-asynchronization}
 For almost every point $(x_1(0), x_2(0))$ in the phase space, there exists an $n = n(x_1(0),x_2(0))$ such that $(x_1(n),x_2(n))\in B_{\gamma_0}$ with $\gamma_0=2^{-20}$.
 \end{Theorem}
 {\it Proof of {\rm(}\ref{disorder-part}{\rm)} by Theorem \ref{equal-asynchronization}}
\qquad Let $$S_0=\{(x_1(0), x_2(0))\in [0, 1]^2\mid (x_1(i), x_2(i))\in G_{\gamma_0},\  \text{for all}\ i\}.$$ Then
 from Theorem \ref{equal-asynchronization}, we have that $\bm(S_0)=0$. Let $$ S_n = \{ (x_1(0), x_2( 0 ) ) \in [0, 1]^2 \mid (x_1(i), x_2(i)) \in G_{\gamma_0},\  \text{for all}\ i>n \} $$ be the subset of $[0, 1]^2$ such that for each point $\bp $ in it and each $i>n$, $T^i(\bp )$ always stay in $G_{\gamma_0}$. From the definition, we have that
$T^{n+1}(S_n)\subset S_0$. If $\bm(S_n)>0$, then there exists $\widetilde{S}_n\subset S_n$ with $\bm(\widetilde{S}_n)>0$ such that $T^{n+1}: \widetilde{S}_n\rightarrow T^{n+1}(\widetilde{S}_n)\subset S_0$ is a diffeomorphism. It implies
$\bm(S_0)\ge \bm(T^{n+1}(\widetilde{S}_n))>0$. This contradicts  the fact that $\bm(S_0)=0$. Hence
 $\bm(S_n)=0$ for each $n$, which leads to {\rm(}\ref{disorder-part}{\rm)}.\hfill\qedbox

 From Fubini's Theorem, one can easily see that Theorem \ref{equal-asynchronization} can be reduced to the following statement:
%\begin{Lemma}\label{secondmainlemma} {\it (The second main lemma)}\par
 for almost every segment with a slope $\pm1$ in the phase space, almost every point on it will be mapped into $B_{\gamma_0}$ in a finite time. Thus, it is sufficient to prove the following lemma.
\begin{Lemma}\label{asynchronization-segment}
There exists a constant $c_1>0$ such that for any $
0\le c<c_1$, there exists $0\leq c_3\equiv c_3(c)<1$ such that for almost each segment $\Gamma$ with a slope $\pm 1$ in $[0, 1]^2$, there exist its disjoint subsegments $\Gamma_1, \Gamma_2, \cdots$
 satisfying {\rm (i)} for any $i$, there exists $l_i\ge 0$ such that
 $T^{l_i}(\Gamma_i)\subset B_{\gamma_0}$; {\rm (ii)}
 $\bm(\cup_i\Gamma_i)\ge c_3  \bm(\Gamma)$.

%There exists a constant $0<c_3<1 $ such that for almost every horizontal or vertical segment $L$ in the phase space, there exists a collection of disjoint sub-segments $L_i\subset L,\ i=1, 2, \cdots$ with $\sum_i|L_i|\ge c_3|L|$ such that each $L_i$ will be mapped into $B_{\gamma_0}$ in a finite time.
\end{Lemma}
The proof of Lemma \ref{asynchronization-segment} can be divided into two propositions.
 \begin{Proposition}\label{step1} There exists a constant $c_1>0$ such that for any $
0\le c<c_1$, there exist two constants $c_4, \delta_2>0$ with the following properties: for almost every segment $\Gamma$  with a slope $\pm1$ in some small square of $[0,1]^2$, there exist its subsegments $\Gamma_1, \Gamma_2, \cdots$ satisfying {\rm (i)} for any $i$, there exists $l(\Gamma_i)\ge 0$ such that $T^{l(\Gamma_i)}(\Gamma_i)$ is a component of $T^{l(\Gamma_i)}(\Gamma)$ and $T^{l(\Gamma_i)}(\Gamma_i)$ is a segment with a slope $\pm 1$ and a length larger than  $\delta_2$, or $T^j(\Gamma_i)\subset B_{\gamma_0}$ for some $j\le l(\Gamma_i)$; {\rm (ii)}
 $\bm(\cup_i\Gamma_i)\ge c_4  \bm(\Gamma)$.

 % with $\sum_i|L_i|\ge c_4|L|$ such that for each $L_i$, there exists $k_i\in N$ such that $T^{k_i}L_i$  is a
%horizontal or vertical segment longer than $\delta_2$ or $T^jL_i\subseteq G_{\gamma_0}$ for some $j\le k_i$.
\end{Proposition}

 \begin{Proposition}\label{step2}  \ Let $c_1, \ c,\ \delta_2$ be defined as above. Assume $\Gamma_0$ is a
segment with a slope $\pm1$ longer than $\delta_2$ in some small square of $[0,1]^2$. Then there exist $l\ge 0$ and disjoint subsegments $\widehat{\Gamma}_1, \widehat{\Gamma}_2,\cdots $ of $\Gamma_0$ with $\sum_i\bm(\widehat{\Gamma}_i)\ge \frac{1}{2} \bm(\Gamma_0)$ such that $T^l(\widehat{\Gamma}_i)\subset B_{\gamma_0}$ for each $i$. \end{Proposition}

\noindent {\it Proof of Proposition \ref{step1}}\qquad
  Recall that $G_{\gamma}=\{\bx \in [0,1]^2\left| \mathrm{dist} (\bx, D_{\mathrm{syn}})\le \gamma\right.\}$ for $\gamma>0$, where $D_{\mathrm{syn}}=\{\bx\in [0,1]^2 | x_1=x_2\} $. Let $\widetilde{G}_{\gamma}=\{\bx\in G_{\gamma}\left|
\mathrm{dist} (\bx, \{x_1+x_2=1\})\leq \gamma\right.\}$. Define $\delta_2=2^{-8}$.  Obviously, if $T(\bx)$ has multi-preimages with $\bx\in G_{\gamma_0}$, then $\bx\in \widetilde{G}_{\gamma_0}$.

 Let $\Omega$ be a segment in $G_{\gamma_0}\cap\{{\rm \ some\ small\ square \ in\ }  [0,1]^2\}$ with a length smaller than $\delta_2$ such that $T^l(\Omega)\nsubseteq D_{\mathrm{syn}}$ for any $l$. We claim that
\vskip 0.2cm
{\bf Claim.}\quad
 {\it For a segment $\Omega\subset G_{\gamma_0}$ stated above, $T^4(\Omega)$ has at most two components.}
\vskip 0.2cm

Without loss of generality, suppose $T(\Omega)$ has two components. Then we have $\Omega\cap \widetilde{G}_{\gamma_0}\not=\emptyset$. Note that $(1/2, 1/2)\in \widetilde{G}_{\gamma_0}$ and the diameter of $\widetilde{G}_{\gamma_0}$ is less than  $2^{-19}$. From the fact that the spectral radius of $T$ is not larger than 2, we obtain that the length of $T(\Omega)$ is less than $2\bm(\Omega)\le 2\delta_2=2^{-7}$, which implies $T(\Omega)$ is in the $2^{-6}$-neighborhood of $(1/2, 1/2)$.

 Moreover, since $T^i(1/2,1/2)=(0,0){\rm (mod\ 1)}$ for $i\ge 1$, we have that $T^{i+1}(\widetilde{G}_{\gamma_0})$ is in $2^{-2}$-neighborhood of $(0,0)$ for $1\le i\le 4$,  which is far from $\widetilde{G}_{\gamma_0}$. Then we obtain the claim.

\vskip 0.2cm
 Thus, from the claim and similar to the proof of Proposition \ref{firststep-L2.2} for the ordered part, we obtain the existence of a constant $c_4$ and subsegments $\Gamma_i$ of $\Gamma$ with a total measure larger than $c_4  \bm(\Gamma)$ such that for each $\Gamma_i$, there exists $l_i\ge 0$ such that $T^{l_i}(\Gamma_i)$  is a
segment with slope $\pm1$ longer than $\delta_2$ or there exists some $j\le l_i$ such that $T^j(\Gamma_i)\subseteq B_{\gamma_0}$. Since the measure of preimages of $D_{\mathrm{syn}}$ is zero, the conclusion is obtained.
\hfill\qedbox\par

\begin{Remark}
Note that there may be a `triple fold' for $T^i(\Omega)$, i.e., $T^i(\Omega)$ may have intersection points with the lines $x_1=1/2$, $x_2=1/2$ and $x_1+x_2=1$ simultaneously, thus $T^i(\Omega)$ consists of 4 segments. In spite of this, the argument above is still valid.
\end{Remark}

\noindent {\it Proof of Proposition \ref{step2}}
\qquad If the  segment $\Gamma_0$ is of slope $-1$, since $\bm(\Gamma_0)>\delta_2>2\gamma_0$, the proof is trivial. Thus we assume $\Gamma_0$ is of slope $1$. Denote  $\mathrm{d}_{\mathrm{syn}} (\Gamma_0)\equiv \mathrm{dist} (\Gamma_0,D_{\mathrm{syn}})$ and let $N(\Gamma_0)=\left\lfloor\log_{2(1-2c)}\frac{\gamma_0}{\mathrm{d}_{\mathrm{syn}} (\Gamma_0)}\right\rfloor$. We say  $\bx \in \Gamma_0$ is a $i$-regular point if $T^j(\bx)\not\in {\widetilde{G}}_{2^j(1-2c)^j\mathrm{d}_{\mathrm{syn}} (\Gamma_0)}$ for each $j\le i$. For simplicity, we assume that
$\Gamma_0\subset \{\bx\in [0,1]^2 | x_2-x_1\le -\mathrm{d}_{\mathrm{syn}} (\Gamma_0)\}$.

Denote \begin{align*}&\Gamma_{1,m} =T(\Gamma_0)\cap {\widetilde{G}}_{2(1-2c)\mathrm{d}_{\mathrm{syn}} (\Gamma_0)},\\
&\Gamma_{1,l}=T(\Gamma_0)\cap \{\bx \not\in {\widetilde{G}}_{2(1-2c)\mathrm{d}_{\mathrm{syn}} (\Gamma_0)} | x_1,x_2\le 1/2\},\\
&\Gamma_{1,r}=T(\Gamma_0)\cap \{\bx \not\in  {\widetilde{G}}_{2(1-2c)\mathrm{d}_{\mathrm{syn}} (\Gamma_0)} | x_1,x_2\ge 1/2\}.\end{align*}
Obviously, $\Gamma_1=\Gamma_{1,l}\cup\Gamma_{1,r}$ and  $\Gamma_{1,m}$, are the image of $1$-regular and non-$1$-regular points in $\Gamma_0$ under $T$, respectively. Moreover, $\bm(T(\Gamma_0))\ge 2\delta_2$. Denote $\mathrm{d}_{\mathrm{syn}} (\Gamma_1)\equiv \mathrm{dist} (\Gamma_1,D_{\mathrm{syn}})$ which is $2(1-2c)\mathrm{d}_{\mathrm{syn}} (\Gamma_0)$, since $\Gamma_1$ consists of segments with  slopes $1$. Since $\Gamma_{1,m}\subset \widetilde{G}_{\mathrm{d}_{\mathrm{syn}} (\Gamma_1)}$, we have
$$
\bm(\Gamma_{1,m})\le 2(2(1-2c))\mathrm{d}_{\mathrm{syn}} (\Gamma_0).
$$
See Figure \ref{fig:4.1} for details.

\begin{figure}[H]
\centering
\includegraphics[height = 8 cm, width = 10 cm]{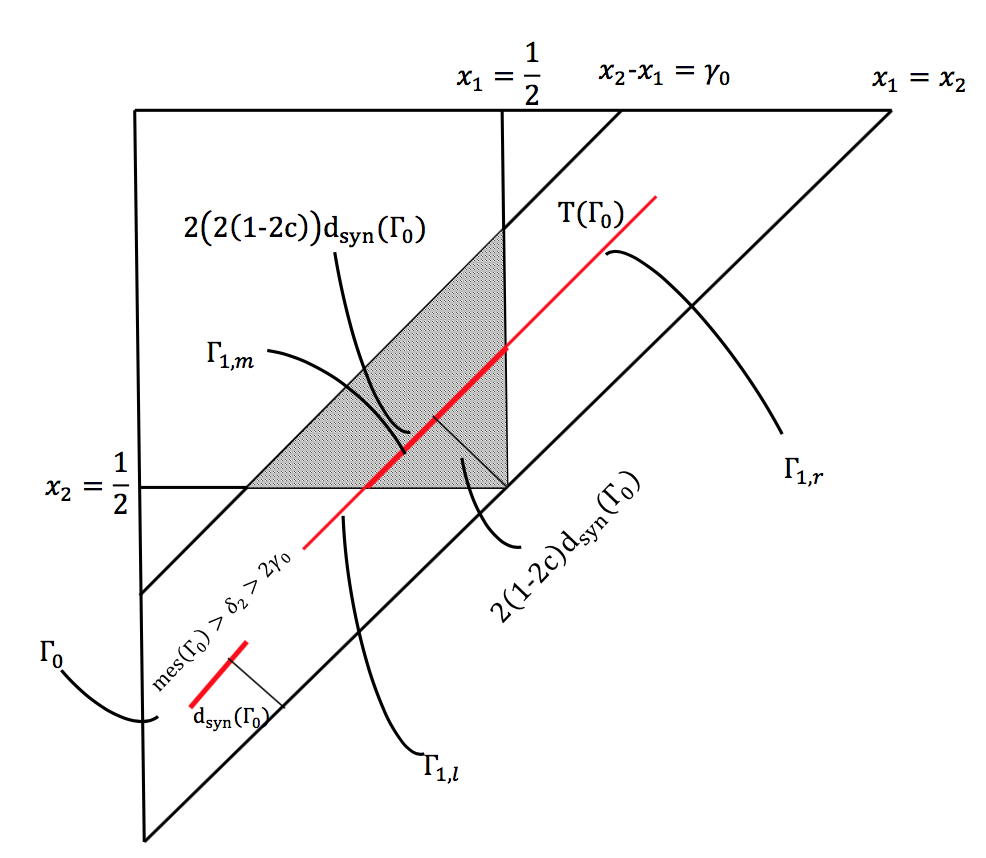}
\caption{$T(\Gamma_0)=(\Gamma_{1,l}\cup \Gamma_{1,r})\cup \Gamma_{1,m}\equiv \Gamma_1\cup\Gamma_{1,m}$ }
\label{fig:4.1} % ÓÃÓÚœ»²æÒýÓÃ
\end{figure}
Thus, the ratio $r_1$ of $1$-regular points in $\Gamma_0$ is  larger than $1-\frac{2(2(1-2c))\mathrm{d}_{\mathrm{syn}} (\Gamma_0)}{\bm(T(\Gamma_0))} \geq \tilde{r}_1\equiv1-\frac{2(1-2c)\mathrm{d}_{\mathrm{syn}} (\Gamma_0)}{\delta_2}$, and $\bm(\Gamma_1)\geq r_1  \bm(T(\Gamma_0))\ge 2\delta_2\tilde{r}_1$.

Similarly, denote \begin{align*}&\Gamma_{2,m}=T(\Gamma_1)\cap {\widetilde{G}}_{2(1-2c)\mathrm{d}_{\mathrm{syn}} (\Gamma_1)},\\
&\Gamma_{2,l}=T(\Gamma_1)\cap \{\bx \not\in  {\widetilde{G}}_{2(1-2c)\mathrm{d}_{\mathrm{syn}} (\Gamma_1)} | x_1,x_2\le 1/2\},\\
&\Gamma_{2,r}=T(\Gamma_1)\cap \{\bx \not\in {\widetilde{G}}_{2(1-2c)\mathrm{d}_{\mathrm{syn}} (\Gamma_1)} | x_1,x_2\ge 1/2\}.\end{align*}
Then $\Gamma_2=\Gamma_{2,l}\cup\Gamma_{2,r}$ and $\Gamma_{2,m}$, are the images of $1$-regular and non-$1$-regular points in $\Gamma_1$ under $T$, respectively. It is clear that the distance $\mathrm{d}_{\mathrm{syn}} (\Gamma_2)\equiv\mathrm{dist} (\Gamma_2,D_{\mathrm{syn}})$ is $(2(1-2c))^2\mathrm{d}_{\mathrm{syn}} (\Gamma_0)$.

Note that the image of each component of $\Gamma_1$ has at most three components. Thus $\Gamma_2$ or $\Gamma_{2,m}$ have at most four or two components, respectively. Since $\Gamma_{2,m}\subset \widetilde{G}_{\mathrm{d}_{\mathrm{syn}} (\Gamma_2)}$, it holds that  $\bm(\Gamma_{2,m})\le 2^2(2(1-2c))^2\mathrm{d}_{\mathrm{syn}} (\Gamma_0)$. Obviously, $\bm(T(\Gamma_{1}))=2\bm(\Gamma_1)$.

Thus the ratio of $1$-regular points in $\Gamma_{1}$ is larger than $1-\frac{2^2(2(1-2c))^2\mathrm{d}_{\mathrm{syn}} (\Gamma_0)}{\bm(T(\Gamma_1))}$. It implies that the ratio $r_2$ of $2$-regular points in $\Gamma_0$ is larger than
\begin{align*}&r_1  \left(1-\frac{2^2(2(1-2c))^2\mathrm{d}_{\mathrm{syn}} (\Gamma_0)}{\bm(T(\Gamma_1))}\right)
\ge r_1-r_1  \frac{2^2(2(1-2c))^2\mathrm{d}_{\mathrm{syn}} (\Gamma_0)}{4\delta_2  r_1}\\
\\=&r_1-\frac{(2(1-2c))^2\mathrm{d}_{\mathrm{syn}} (\Gamma_0)}{\delta_2}\geq1-[2(1-2c)+(2(1-2c))^2]\frac{\mathrm{d}_{\mathrm{syn}} (\Gamma_0)}{\delta_2}\equiv\tilde{r}_2\end{align*}
and $$
\bm(\Gamma_2)\ge \bm(T(\Gamma_1))  \tilde{r}_2\geq2^2\delta_2\tilde{r}_2.$$

Inductively, let $\tilde{r}_{i-1}=1-\sum_{j=1}^{i-1}(2(1-2c))^j\frac{\mathrm{d}_{\mathrm{syn}} (\Gamma_0)}{\delta_2}$ and $\Gamma_{i-1}$ is the segments with slopes $1$ of ($i-1$)-regular points in $\Gamma_0$ satisfying $\bm(\Gamma_{i-1})\ge 2^{i-1}\delta_2\tilde{r}_{i-1}  $ and the number of segments in $\Gamma_{i-1}$ is not more than $2^{i-1}$. Moreover, the distance $\mathrm{d}_{\mathrm{syn}} (\Gamma_{i-1})\equiv \mathrm{dist} (\Gamma_{i-1},D_{\mathrm{syn}})$ is $(2(1-2c))^{i-1}\mathrm{d}_{\mathrm{syn}} (\Gamma_0)$.

Then $\bm(T(\Gamma_{i-1}))=2\bm(\Gamma_{i-1})\ge 2^i \delta_2 \tilde{r}_{i-1}  $ and the set of 1-regular and non-1-regular points in $\Gamma_{i-1}$ has at most $2^{i}$ and $2^{i-1}$ components, respectively. Obviously the distance $\mathrm{d}_{\mathrm{syn}} (\Gamma_i)\equiv (2(1-2c))^{i}\mathrm{d}_{\mathrm{syn}} (\Gamma_0)$.

Let $\Gamma_i$ and $\Gamma_{i,m}$ be the image of 1-regular and non-1-regular points in $\Gamma_{i-1}$, respectively. Since $\Gamma_{i,m}\subset \widetilde{G}_{\mathrm{d}_{\mathrm{syn}} (\Gamma_i)},$ we have that the ratio of $1$-regular points in $\Gamma_{i-1}$ is larger than $1-\frac{2^{i}(2(1-2c))^i\mathrm{d}_{\mathrm{syn}} (\Gamma_0)}{2^{i} \delta_2\tilde{r}_{i-1} }$.
It implies the ratio of $i$-regular points in $\Gamma_0$ is larger than
\begin{align*}&r_{i-1}  (1-\frac{(2(1-2c))^i\mathrm{d}_{\mathrm{syn}} (\Gamma_0)}{\delta_2\tilde{r}_{i-1}})
\ge \tilde{r}_{i-1}- \frac{(2(1-2c))^i\mathrm{d}_{\mathrm{syn}} (\Gamma_0)}{\delta_2}\\
=&1-\sum_{j=1}^{i}(2(1-2c))^j\frac{\mathrm{d}_{\mathrm{syn}} (\Gamma_0)}{\delta_2}.\end{align*}
Thus
the ratio of $N(\Gamma_0)$-regular points in $\Gamma_0$ is larger than
$$
r_{_{N(\Gamma_0)}}=1-\sum_{j=1}^{N(\Gamma_0)}(2(1-2c))^j\frac{\mathrm{d}_{\mathrm{syn}} (\Gamma_0)}{\delta_2}.
$$
From the definition of $N(\Gamma_0)$ and the fact that $c$ is small, we have that
\begin{align*}
1-{r}_{_{N(\Gamma_0)}}&\le \frac{2(1-2c)\left(1-(2(1-2c)\right)^{N(\Gamma_0)})}{1-2(1-2c)}  \frac{\mathrm{d}_{\mathrm{syn}} (\Gamma_0)}{\delta_2}\le 4  (2(1-2c))^{N(\Gamma_0)}  \frac{\mathrm{d}_{\mathrm{syn}} (\Gamma_0)}{\delta_2} \\
\\
&\le 4  \frac{\gamma_0}{\mathrm{d}_{\mathrm{syn}} (\Gamma_0)}  \frac{\mathrm{d}_{\mathrm{syn}} (\Gamma_0)}{\delta_2}=\frac{4\gamma_0}{\delta_2}\le 1/2\quad {\rm \ if \ \gamma_0<\frac{\delta_2}{8}.}
\end{align*}
Then we know that $r_{N(\Gamma_0)}>1/2$. Moreover, for each $N(\Gamma_0)$-regular point $\bx $ in
$\Gamma_0$, it is not difficult to see that
$
T^{N(\Gamma_0)+1}(\bx)
$ is in $B_{\gamma_0}$, since the distance between $
T^{N(\Gamma_0)+1}(\bx)
$ and $D_{\mathrm{syn}}$ is larger than $(2(1-2c))^{N(\Gamma_0)+1}  \mathrm{d}_{\mathrm{syn}} (\Gamma_0)\ge \gamma_0$.
%where $N(\Gamma)=\log_{1/(1-2c)}\frac{\gamma_0}{\mathrm{dist} (\Gamma)}$. For small $\gamma_0$(smaller than $c$),
%The measure(may be in the sense of one dim, i.e., the length) for the union of all pre-images of the bad triangle region is bounded by
%$\sum_{i=1}^{N(\Gamma)} (\frac{1}{1-2c})^i *$ the sum will be less than the constant length.
\hfill \  \qedbox\par
\par
\par
\par
\par
\par
\par
\par
\par
\par

\section{Piecewise expanding case}
In this section, we will prove Theorem \ref{nonlinear}.
Recall that the proof of Theorem \ref{main-theorem} depends heavily on the piecewise-linearity of $T$. In fact, it implies the property that  the image of a segment in some small square by $T$ is still a segment, by which the proof can be reduced to the simple fact  that a long enough segment in $[0, 1]^2$ has a nonempty intersection with the line $x_1=x_2$. Unfortunately, this property is not valid any more with the existence of nonlinear perturbation and we have to deal with curves rather than segments. For a general smooth simple  curve $\Gamma$ in $[0, 1]^2$, no matter how long it is, it may occur that
$\Gamma\cap G_{\epsilon}=\emptyset$. To prove Theorem \ref{nonlinear}, we need to exclude the possibility for this troublesome situation. More precisely, we will show that components of a short segment consist of `very flat'  simple curves until their length  are of constant scale (see case (1iii), (2iii) or (kiii) in the proof of Lemma \ref{flat}). Then everything valid to segments sated above will be also valid to `very flat'  simple curves in a similar way.

For this purpose, we need to introduce some quantity to measure how flat a  simple curve is. For a point $\bp  $ in a  simple curve $\Gamma$ where the tangent line can be defined, we denote the unit tangent vector of $\Gamma$ at $\bp $ by $\mathbf{t}_{_{\Gamma}}(\bp )\in \RR^2$ coinciding with an orientation of the curve. Then we define the range of angles on $\Gamma$ to be $\mathrm{r_a}(\Gamma)\equiv\sup_{\bp _1,\bp _2\in\Gamma} \| \mathbf{t}_{_{\Gamma}}(\bp _2)-\mathbf{t}_{_{\Gamma}}(\bp _1) \| $.
\begin{Remark}\label{almost-everywhere}
Let $[0, 1]^2=\cup \Gamma_a$ is a union of segments, where $\Gamma_a$ is a segment with $a$ in some interval $I$. Let $D_{\infty}=\cup_{l=0}^{\infty}T^{-l}(D_{\mathrm{syn}})$. Then the measure of $D_{\infty}$ in $[0, 1]^2$ is zero. Hence from Fubini's theorem, we have that for almost all $a$ in $I$, $T^l(\Gamma_a)\cap D_{\infty}$ is a set of measure zero in $T^l(\Gamma_a)$ for any $l$ and in particular, $T^{-l}(T^l(\Gamma_a)\cap D_{\infty})\cap \Gamma_a=\Gamma_a\cap T^{-l}( D_{\infty})$ does not include any open interval. For this reason, each  simple curve considered in this section satisfies that the intersection set of it with $D_{\infty}$ is of measure zero.
\end{Remark}

For any $\bp \in [0, 1]^2$, let $JT(\bp )$ be the Jacobian matrix of $T$ at $\bp $. Then $JT$ is piecewise $C^1$ on $\bp $. The ordered part in Theorem \ref{nonlinear} can be reduced to the following three lemmas.

\begin{Lemma}\label{flat}
There exist constants $c_1,\eta>0$ such that if $0\le c<c_1,\ \|g\|_{C^2}\le\eta$, there exist constants $c_5>0$ and $0<\hat{a}=O(\eta)$ such that if $\Gamma_0$ is a segment in one of small squares in $[0,1]^2$ with a slope $\pm 1$ and a length less than $\delta_1$, we can find a collection of sub-curves denoted by $\bar{\Gamma}_i$ satisfying
{\rm (i)} $\sum_{i>0} \bm(\bar{\Gamma}_i)\ge c_5\bm(\Gamma_0)$; {\rm (ii)}
for any $i>0$, there exists $l_i$ such that

{\rm (a)} $T^{l_i}(\bar{\Gamma}_i)$ is a component of $T^{l_i}(\Gamma_0)$ and $\bm(T^{l_i}(\bar{\Gamma}_i))\ge \delta_1$;

 {\rm (b)} $\mathrm{r_a}(T^{l_i}(\bar{\Gamma}_i))\le \hat{a}\bm(T^{l_i}(\bar{\Gamma}_i))$.
\end{Lemma}
\begin{proof}
Since $T$ is a small perturbation of a (piecewise) linear map
  satisfying $$JT(\bp)=\begin{bmatrix}  1-c+\eta_1 &c+\eta_2\\ c+\eta_3 & 1-c+\eta_4\end{bmatrix}$$ with $| \eta_i|\leq \eta$, for $i=1,\cdots,4$, it is easily seen that $|| JT(\bp) || \leq 1+O(c,\eta)$.

\begin{figure}[H]
\centering
\includegraphics[height = 8 cm, width = 11 cm]{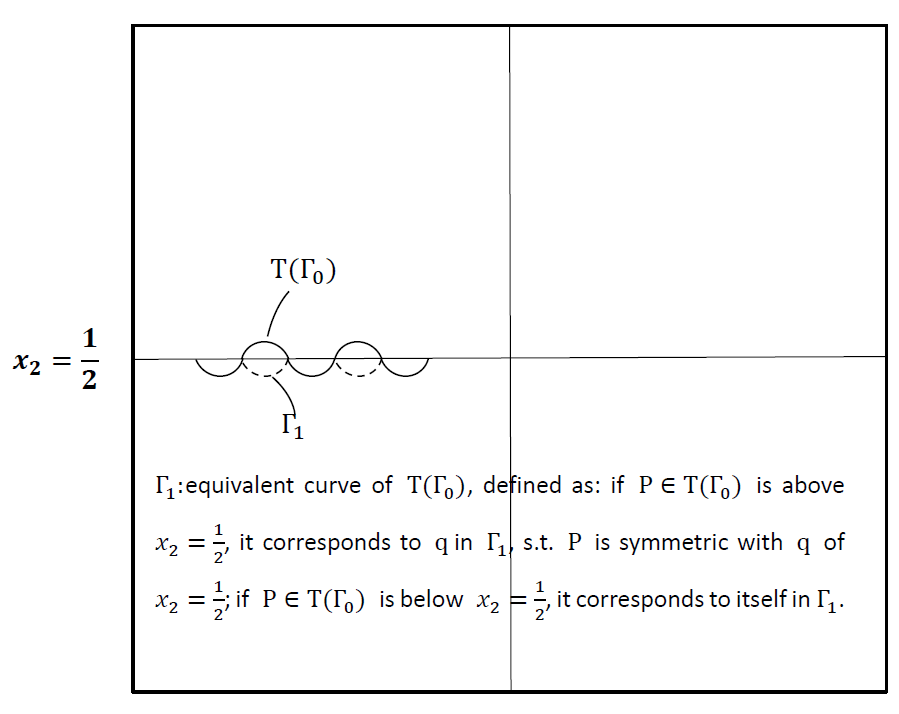}
\caption{Equivalent curve $\Gamma_1$ of $\Gamma_0$.}
\label{fig:5.2} % ÓÃÓÚœ»²æÒýÓÃ
\end{figure}

Let $\Gamma_0$ be a (short) segment as above. Then $\mathrm{r_a}(\Gamma_0)=0$. With the condition $\|g\|_{C^2}$ sufficiently small, it is obvious that if both $x_1=1/2$ and $x_2=1/2$ have intersections with $T(\Gamma_0)$, then the lemma is immediately proved. Hence, without loss of generality, we assume that
$T(\Gamma_0)$ has no intersection with $x_1=1/2$.

 There are three different cases according to the intersection between $T(\Gamma_0)$ and $x_2=1/2$.

{\bf Case (1i)}\ There is no intersection between $T(\Gamma_0)$ and $x_2=1/2$. Denote $\Gamma_1=T(\Gamma_0)$. Since $T$ is a small perturbation of a (piecewise) linear map and $\bm(\Gamma_0)$ is small, we have that $\mathrm{r_a}(\Gamma_1)\le a  \bm(\Gamma_0)$ with $0<a<1$, which is small if $\eta$ is small. In fact, since $\Gamma_0$ is a segment, it holds that $\mathbf{t}_{_{\Gamma_0}}(\bp )$ is constant. Then we have
$$
\|JT(\bp _2)-JT(\bp _1)\|\le \max_{p\in \Gamma_0}\|D(JT)(p)\|  \|\bp _2-\bp _1\|\le a  \|\bp _2-\bp _1\|,
$$
where $D(JT)$ is the Jacobian matrix of $JT$ with respect to $\bp $
and $a=O(\eta)\ll 1$ if $\eta\ll 1$.

Subsequently, because of $\mathbf{t}_{_{\Gamma_0}}(\bp _1)=\mathbf{t}_{_{\Gamma_0}}(\bp _2)$, we obtain that
\begin{align*}
&\|\mathbf{t}_{_{\Gamma_1}}(T(\bp _2))-\mathbf{t}_{_{\Gamma_1}}(T(\bp _1))\|=\| JT(\bp _2)\mathbf{t}_{_{\Gamma_0}}(\bp _2)-JT(\bp _1)\mathbf{t}_{_{\Gamma_0}}(\bp _1) \|
\\ \le& \|JT(\bp _2)-JT(\bp _1)\|\le a  \|\bp _2-\bp _1\|
\end{align*}
with $a=O(\eta)$ small.

{\bf Case (1ii)}\ There is exactly one intersection point between $T(\Gamma_0)$ and $x_2=1/2$. Denote two parts of $T(\Gamma_0)$ divided by the intersection point by $\Gamma_{1,1}$ and $\Gamma_{1,2}$, respectively. Obviously, it holds that
$\mathrm{r_a}(\Gamma_{1,i})\le a  \bm(\Gamma_{1,i}),\ i=1,\ 2$ with the same $a$ as in case (1i). Let $\Gamma_1$ be any one of these two  simple curves, say, $\Gamma_{1,1}$.

 {\bf Case (1iii)}\ There are more than one intersection points. Same as in case (1i), we have $\mathrm{r_a}(T(\Gamma_0))\le  a \bm(\Gamma_0)$. Although $T(\Gamma_0)$ has 3 or more components, in the following iterations, it can be replaced by a simple curve in some small square with a small range of angle as follows (note that we have to consider different components separately only if the range of the curve is not small).  Define
  $\Gamma_1$ be a  piecewise-smooth simple curve in some small square satisfying that
  $T(\Gamma_1)=T^2(\Gamma_0)$ (see Figure \ref{fig:5.2}). Since there exist at least two intersection points between $T(\Gamma_0)$ and $x_2=1/2$, there exist $\bp_1, \bp_2\in T(\Gamma_0)$ such that $\mathbf{t}_{T(\Gamma_0)}(\bp_1)$ and $\mathbf{t}_{T(\Gamma_0)}(\bp_2)$ lies in the upper and lower half planes, respectively.  Thus for any point $\bp $ in $T(\Gamma_0)$, it holds that
  $ \| \mathbf{t}_{T(\Gamma_0)}(\bp )\|\le \mathrm{r_a}(T(\Gamma_0)) \leq  a  \bm(\Gamma_0)$. Then for any  points $\bp _1, \bp _2$ in $\Gamma_1$, we have
  $ \| \mathbf{t}_{_{\Gamma_1}}(\bp _2)-\mathbf{t}_{_{\Gamma_1}}(\bp _1)\le  \| \mathbf{t}_{\Gamma_1}(\bp _2)\| +\| \mathbf{t}_{\Gamma_1}(\bp _1) \| \le  2a  \bm(\Gamma_0)$. Therefore $\mathrm{r_a}(\Gamma_1)\le 2a\bm(\Gamma_0)$.

%\begin{figure}[!htb]
%\centering
%%\includegraphics[height = 8cm, width = 10 cm]{report_pic_16.jpg}
%\caption{ÍŒÆ¬ÏÂÃæµÄÎÄ×ÖËµÃ÷.}
%\label{fig:XXXX} % ÓÃÓÚœ»²æÒýÓÃ
%\end{figure}

 Next we consider three different cases for $T(\Gamma_1)$ according to the intersection between $T(\Gamma_1)$ and $x_2=1/2$.

{\bf  Case (2i)}\ There is no intersection. Denote $\Gamma_2=T(\Gamma_1)$.
  For $\bm(\Gamma_0)$ small, and for any points $\bp_1,\bp_2$ in $\Gamma_1$, we then have
%  and $\bm(\Gamma_0)$ is small, we have that $\mathrm{r_a}(\Gamma_2)\le a  \bm(\Gamma_1)+(\frac{2}{\lambda}+\eta)\mathrm{r_a}(\Gamma_1)$, where $\lambda=2(1-2c)$. In fact, for any points $\bp _1, \bp _2$ in $\Gamma_1$, it holds that
 \begin{align*}
 & \| \mathbf{t}_{_{\Gamma_2}}(T(\bp _2))-\mathbf{t}_{_{\Gamma_2}}(T(\bp _1)) \| =||JT(\bp _2)\mathbf{t}_{_{\Gamma_1}}(\bp _2)-JT(\bp _1)\mathbf{t}_{_{\Gamma_1}}(\bp _1)||\\
 \le &||JT(\bp _2)\mathbf{t}_{_{\Gamma_1}}(\bp _1)-JT(\bp _1)\mathbf{t}_{_{\Gamma_1}}(\bp _1)||+||JT(\bp _2)\mathbf{t}_{_{\Gamma_1}}(\bp _2)-JT(\bp _2)\mathbf{t}_{_{\Gamma_1}}(\bp _1)||\\
\le  &a  \bm(\Gamma_1)+(1+O(c,\eta)) \| \mathbf{t}_{_{\Gamma_1}}(\bp _2)-\mathbf{t}_{_{\Gamma_1}}(\bp _1) \| \\
\le  &a  \bm(\Gamma_1)+(1+O(c,\eta))\mathrm{r_a}(\Gamma_1).
 \end{align*}
Therefore, it holds that $\mathrm{r_a}(\Gamma_2)\leq a \bm(\Gamma_1)+(1+O(c,\eta))\mathrm{r_a}(\Gamma_1)$.

{\bf Case (2ii)}\ There is exactly one intersection point. Assume $T(\Gamma_1)=\Gamma_{2,1}\cup\Gamma_{2,2}$ and $\Gamma_{2,1}\cap\Gamma_{2,2}$ be the intersection point. For each $\Gamma_{2,i}$, we have a similar estimate for $\mathrm{r_a}(\Gamma_{2,i})$ as in case(2i). We will denote any component of it, say $\Gamma_{2,1}$, by $\Gamma_2$.

 {\bf Case (2iii)}\ There are more than one intersection points. Same as in case (2i), we have $\mathrm{r_a}(T(\Gamma_1))\le  a  \bm(\Gamma_1)+(1+O(c,\eta))\mathrm{r_a}(\Gamma_1)$.  Since in the  simple curve there are at least two points on $x_2=1/2$,
 there exist $\bp_1, \bp_2\in T(\Gamma_1)$ such that $\mathbf{t}_{T(\Gamma_1)}(\bp_1)$ and $\mathbf{t}_{T(\Gamma_1)}(\bp_2)$ lies in the upper and lower half planes, respectively.  Thus for any point $\bp $ in $T(\Gamma_1)$, it holds that
  $ \| \mathbf{t}_{T(\Gamma_1)}(\bp )\|\le \mathrm{r_a}(T(\Gamma_1)) \leq  a  \bm(\Gamma_1)+(1+O(c,\eta))\mathrm{r_a}(\Gamma_1)$.
 Let $\Gamma_2$ be a  piecewise-smooth simple curve in some small square satisfying $T(\Gamma_2)=T^2(\Gamma_1)$. Thus for any  points $\bp _1, \bp _2$ in $\Gamma_2$, we have $$ \begin{array}{ll}\| \mathbf{t}_{_{\Gamma_2}}(\bp _2)-\mathbf{t}_{_{\Gamma_2}}(\bp _1) \|
  &\le  \| \mathbf{t}_{\Gamma_2}(\bp _2) \| + \|  \mathbf{t}_{\Gamma_2}(\bp _1)  \|\\
  \\
  & \le  2a  \bm(\Gamma_1)+2(1+O(c,\eta))\mathrm{r_a}(\Gamma_1).\end{array}$$

Note that $\lambda=2(1-2c)$. By induction, we have that for any $k$, in case (ki) and (kii) it holds that \begin{align} \label{ki}
\mathrm{r_a}(\Gamma_k)&\le a  \bm(\Gamma_{k-1})+(1+O(c,\eta))\mathrm{r_a}(\Gamma_{k-1}) \nn \\
&\le a\lambda^{-1}\bm(\Gamma_k)+(1+O(c,\eta))\mathrm{r_a}(\Gamma_{k-1})\quad {\rm (\ since\ \bm(\Gamma_k)\ge \lambda \bm(\Gamma_{k-1}))},
\end{align}
 and in case (kiii) we have that
\begin{align} \label{near-flat}\mathrm{r_a}(\Gamma_k)\le 2a  \bm(\Gamma_{k-1})+2(1+O(c,\eta))\mathrm{r_a}(\Gamma_{k-1})\le 2a  \lambda^{-1}   \bm(\Gamma_{k})+2(1+O(c,\eta))\mathrm{r_a}(\Gamma_{k-1}).\end{align}
 On the other hand, the frequency for the occurrence of case (kiii) is very low. In fact, in case (kiii), (\ref{near-flat}) together with the fact that $\Gamma_k$ is short imply that the  simple curve $\Gamma_k$  `nearly' coincides with the line $x_1=1/2$, so it is mapped into a  simple curve `nearly' coincides with the line $x_1=1$ by $T$.
 Thus, we may assume that for $j=6l+6$, (jiii) occurs and thus $\mathrm{r_a}(\Gamma_j)$ should be estimated by (\ref{near-flat}), while for $j=6l+1,\cdots,6l+5 $, (jiii) will not occur and $\mathrm{r_a}(\Gamma_j)$ should be estimated by
 (\ref{ki}). Hereafter, we let $g\equiv 1+O(c,\eta)$.  Since $\bm(\Gamma_{6l+i})\leq \lambda^{-(6-i)}\bm(\Gamma_{6l+6})$ and $\mathrm{r_a}(\Gamma_0)=0$, we have
 \begin{align}\label{ki-ii}
&\mathrm{r_a}(\Gamma_{6(l+1)}) \leq 2a\lambda^{-1}\bm(\Gamma_{6(l+1)})+2gr_a
(\Gamma_{6l+5})\leq\cdots \nonumber\\
\leq &2a\lambda^{-1}\bm(\Gamma_{6(l+1)})+2ga\lambda^{-1}\bm
(\Gamma_{6l+5})+2g^2a\lambda^{-1}\bm
(\Gamma_{6l+4})+2g^3a\lambda^{-1}\bm
(\Gamma_{6l+3})\nonumber\\
&+2g^4a\lambda^{-1}\bm
(\Gamma_{6l+2})+2g^5r_a(\Gamma_{6l})\le \cdots\nonumber\\
\leq &2a\lambda^{-1}\left(  1+\frac{g}{\lambda}+\cdots+(\frac{g}{\lambda})^5\right)\left( 1+2(\frac{g}{\lambda})^6+(2(\frac{g}{\lambda})^6)^2+\cdots+(2(\frac{g}{\lambda})^6)^l\right)\bm(\Gamma_{6(l+1)})\nonumber\\
&+(2g^6)^l \mathrm{r_a}(\Gamma_0)\nonumber\\
=&2a\lambda^{-1}\left( \sum_{i=0}^5(\frac{g}{\lambda})^i\right) \sum_{i=0}^l(2(\frac{g}{\lambda})^6)^i \bm(\Gamma_{6(l+1)}).
 \end{align}
 For $k=6l+j,\ j=1,\cdots,5$, let $b_i=(2(\frac{g}{\lambda})^6)^i (\frac{g}{\lambda})^j$. Then from (\ref{ki-ii}), if $2(\frac{g}{\lambda})^6<1$, we have $\mathrm{r_a}(\Gamma_k)\leq \hat{a}\bm(\Gamma_k)$, where $\hat{a}=2a\lambda^{-1}\sum_{i=0}^5 (\frac{g}{\lambda})^i \beta$ with $\beta=\sum_{i=0}^\infty b_i<\infty$.

     With these estimates, we conclude that estimates for current situation is totally similar to the one in the proof for the tent map. In particular, Corollary \ref{Corollary} is available and thus we can obtain the lemma.
 \hfill   \qedbox
\end{proof}

In the following we will give the proof for the case $f_0(x)=1-s|x-1/2|$,  $s=2-s_0>0$ with $s_0>0$.
For the case $s_0=0$, the proof can be obtained in a similar (in fact simpler) way.

 The following lemma can make the argument simpler.
 \begin{Lemma}\label{equivalent-ds}
Assume $0<c, \eta\ll s_0\ll 1$. For almost every point $\bp $ in $[0,1]^2$, there exists a $i(\bp )\in\NN$ such that for all $i\geq i(\bp )$, $x_1(T^i(\bp )),x_2(T^i(\bp ))\in[
\tau_2,1-\tau_1]$, where $\tau_1=\frac{s_0}{2}-\eta>0$ and $\tau_2=(1-c)(s-\eta)(\frac{s_0}{2}-\eta)>0$.
\end{Lemma}
\begin{Proof}
If $0<x_1(\bp )<\frac{1}{4}$, then $x_1(T(\bp ))=(1-c)f(x_1(\bp ))+cf(x_2(\bp ))\geq (s-\eta)(1-c)x_1(\bp )$. Thus $\{x_1(T^i(\bp ))\}_i$ is an increasing sequence if only $x_1(T^i(\bp))\le \frac14$.
If $x_1(\bp )> \frac{3}{4}$, then $x_1(T(\bp ))\leq  \max_{x\in[\frac{3}{4},1]} f(x)<\frac{3}{4}$.

Thus we only need to consider the situation $\frac{1}{4}\leq x_1(\bp )\leq \frac{3}{4}$.  For this case, it holds that
$$ x_1(T(\bp ))=(1-c)f(x_1(\bp ))+cf(x_2(\bp) ))\leq \max_{x\in[0,1]}f(x)\leq \frac{s}{2}+\eta= 1-\frac{s_0}{2}+\eta =1-\tau_1. $$
Subsequently, we have
$$ x_1(T^2(\bp ))\geq (1-c)f(\tau_1)\geq (1-c)(s-\eta)(\frac{s_0}{2}-\eta )=\tau_2.$$

Thus eventually the orbit of $\bp$ under the map $T$ lies between $x_1=\tau_2$ and $x_1=1-\tau_1$.

Similarly we can obtain the estimate for $x_2(\bp)$. This completes the proof. \hfill   \qedbox
\end{Proof}
\vskip 0.2cm
 From Lemma \ref{equivalent-ds}, without loss of generality we can replace the phase space $[0,1]^2$ by $[\tau_2, 1-\tau_1]^2$, that is,  \begin{equation}\label{*} x_1, x_2\in [\tau_2, 1-\tau_1]. \end{equation}

    \begin{Lemma}\label{nonlinear-after-delta1}
   There exists $c_1>0$ such that for any $0\le c<c_1$, there exists a fixed number $c_2>0$ such that
for any piecewise ${\cal C}^2$-smooth simple curve $\Gamma_0$ with a length $\bm(\Gamma_0)\in [\delta_1, 2\delta_1]$ and satisfying $\mathrm{r_a}(\Gamma_0)\le O(\eta ) \bm(\Gamma_0)$, there exists a simple curve  $\Gamma\subset \Gamma_0$ with $\bm(\Gamma)\ge c_2 \bm(\Gamma_0)$ and $l(\Gamma)\in \NN$ satisfying $T^{l(\Gamma)}(\Gamma)\subset G_{\epsilon}$.
       \end{Lemma}

\begin{proof}
It can be reduced to the following claim.

{\bf Claim}. Let $\delta_1=2^{-16}$ defined as before. Then there exist a constant $\hat{e}>1$ and $i_0$ independent of $\eta $, such that for $\eta \ll 1$, the following holds true:
\begin{itemize}
  \item[]
Assume for each $0\le i\le j\le N_0$ with $N_0=\left\lfloor\log_{\hat{e}}2\delta_1^{-1}\right\rfloor+1$, we have inductively defined  piecewise ${\cal C}^2$-smooth simple curves $\Gamma_i$ in some small square with ${\Gamma}_{i}\subset T^{l(i)}(\Gamma_{i-1})$ for some $l(i)\le i_0+3$ ($i\ge 1$), such that
$\bm(\Gamma_i)\geq \hat{e}\bm(\Gamma_{i-1})$
 and $\mathrm{r_a}(\Gamma_i)\leq O(\eta )\bm(\Gamma_i)$.

If for each $0\le i\le j$, $\bm(\Gamma_i)\leq 2$ and $\Gamma_{i}\cap \{x_1=x_2\}\neq \emptyset$, we have that there exist $l(j+1)\le i_0+3$ and a simple curve  ${\Gamma}_{j+1}\subset T^{l(j+1)}(\Gamma_j)$ in some small square such that
%\begin{enumerate}
 % \item[(a)] $\Gamma_{j+1}$  lies in some small square;
   %and $\bm(\widetilde{\Gamma}_{j+1})\geq \frac{1}{2}\bm(\Gamma_j)$;
  (a)$\bm(\Gamma_{j+1})\geq \hat{e}\bm(\Gamma_{j})$;
  (b) $\mathrm{r_a}(\Gamma_{j+1})\leq O(\eta )\bm(\Gamma_{j+1})$.
%\end{enumerate}
\end{itemize}
`Claim $\Rightarrow$Lemma \ref{nonlinear-after-delta1}'.

First we prove the existence of a $j_0\le N_0$ satisfying $\Gamma_{j_0}\cap \{x_1=x_2\}\neq \emptyset$. Otherwise, from the claim we have that either there exists $k_0\le N_0$ such that $\Gamma_i$ is defined for each $0\leq i\leq k_0$ satisfying $\bm(\Gamma_{k_0})\geq 2$ and $\bm(\Gamma_i)<2$ for $0\leq i\leq k_0-1$, or for each $0\leq i \leq N_0$, $\Gamma_i$ is defined with $\bm(\Gamma_{i})<2$.

Note that for any simple curve $\Gamma$ in some small square, it holds that $(\lambda-O(\eta ))\bm(\Gamma)\leq \bm(T(\Gamma))\leq (2+O(\eta ))\bm(\Gamma)$. Then for the former case, from (b) in the claim, we have
 \begin{align*}& \mathrm{r_a}(\Gamma_{k_0})\leq O(\eta )\bm(\Gamma_{k_0}) \leq  O(\eta )\bm(T^{l(k_0)}(\Gamma_{k_0-1}))\\\leq & (2+O(\eta ))^{l(k_0)}O(\eta )\bm(\Gamma_{k_0-1})\leq 2O(\eta )(2+O(\eta ))^{3+i_0}.\end{align*}
 Hence if $O(\eta )\leq 200^{-1}$, we obtain that $\mathrm{r_a}(\Gamma_{k_0})\leq 10^{-1}$. But $\bm(\Gamma_{k_0})\geq 2$ and it is clear that there is no such a simple curve  in $[0,1]^2$.
For the latter case, from (a) in the claim, we know that $\bm(\Gamma_{N_0})\geq \hat{e}^{N_0}\bm(\Gamma_0)\geq \hat{e}^{\left\lfloor\log_{\hat{e}}2\delta_1^{-1}\right\rfloor+1}\delta_1\geq 2$.
This contradicts the assumption that $\bm(\Gamma_{N_0})<2$ and hence we obtain the existence of $j_0$. From the definition of $\Gamma_i$, $i\le j_0$, there exists some $n(j_0)\le (i_0+3)j_0$ such that $\Gamma_{j_0}\subset T^{n(j_0)}(\Gamma_0)$.

 Let $\widetilde{\Gamma}=\Gamma_{j_0}\cap G_{\epsilon}$ and define ${\Gamma}\subset \Gamma_0$ such that $T^{n(j_0)}({\Gamma})=\widetilde{\Gamma}$. First we consider the case that $\Gamma_{j_0}\subset G_{\epsilon}$, that is, $\widetilde{\Gamma}=\Gamma_{j_0}$.
%From (a) in the claim, we know that for each $0\le i\leq j_0$, $\Gamma_i$ is a simple curve in some small square and %$\bm(\widetilde{\Gamma}_i)\geq \frac{1}{2}\bm(\Gamma_{i-1})$ for each $1\le i\leq j_0$. Hence $\bm(\Gamma_i)\geq %(\lambda-\hat{a})\bm(\widetilde{\Gamma}_i)\ge \frac{\lambda-\hat{a}}{2}\bm({\Gamma}_{i-1})$.
From (a) in the claim, we have
$\bm(\Gamma_{j_0})\geq  \hat{e}^{j_0}\bm(\Gamma_{0}).
$
On the other hand, it holds that $\bm(\Gamma)\ge (2+O(\eta ))^{-n(j_0)}\bm(\Gamma_{j_0})$.  Consequently, we obtain
%$$
%\bm(\Gamma_{j_0})\geq \hat{e}^{\left\lfloor\frac{j_0}{3}\right\rfloor} \min\{1, (2+O(\eta ))^{-2}\hat{e}\} \bm(\Gamma_{0}).
%$$
%$$\bm(\Gamma_{j_0})\geq \left( %\frac{\lambda-O(\eta )}{2}\right)^{j_0-3\left\lfloor\frac{j_0}{3}\right\rfloor}\bm(\Gamma_{3\left\lfloor\frac{j_0}{3}\right\rfloor})\geq %\hat{e}^{\left\lfloor\frac{j_0}{3}\right\rfloor}\bm(\Gamma_{0})\left( \frac{\lambda-O(\eta )}{2}\right)^2.$$
\begin{align*}
&\bm({\Gamma})\geq (2+O(\eta ))^{-(i_0+3)j_0} \hat{e}^{j_0}\bm({\Gamma_0})\ge (2+O(\eta ))^{-(i_0+3)N_0}\bm(\Gamma_{0}).
\end{align*}
This leads to the conclusion if $c_2\leq (2+O(\eta ))^{-(3+i_0)N_0}$.

Next we consider the case that $\Gamma_{j_0}\nsubseteq G_{\epsilon}$. Obviously $\bm(\widetilde{\Gamma})\geq \epsilon$. Then we have that
\begin{align*} &\bm({\Gamma})\geq (2+O(\eta ))^{-n(j_0)}\bm(\widetilde{\Gamma})\\ \geq & (2+O(\eta ))^{-(3+i_0)N_0}\epsilon\geq (2+O(\eta ))^{-(3+i_0)N_0}\epsilon (2\delta_1)^{-1}\bm(\Gamma_0).\end{align*}
If $c_2\leq (2+O(\eta ))^{-(3+i_0)N_0}\epsilon (2\delta_1)^{-1}$, we obtain the conclusion.

Thus by setting $c_2=\min\{1, \epsilon (2\delta_1)^{-1}\} (2+O(\eta ))^{-(3+i_0)N_0}$, we finish the proof.
\vskip 0.3cm
`Proof of the claim'.

\vskip 0.2cm
We consider the following cases.
\begin{enumerate}
\item First, we note if both $T(\Gamma_j)\cap \{x_1=\frac{1}{2}\}$ and $T(\Gamma_j)\cap \{x_2=\frac{1}{2}\}$ are nonempty, then $T(\Gamma_j)\cap \{x_1=x_2\}\not=\emptyset$ or $T(\Gamma_j)\cap \{x_1+x_2=1\}\not=\emptyset$ and hence $T^2(\Gamma_j)$ is nonempty, which implies Lemma \ref{nonlinear-after-delta1} from the argument above. Hence in the following, we will omit the proof for this trivial case and other similar ones (e.g., both $T^2(\Gamma_j)\cap \{x_1=\frac{1}{2}\}$ and $T^2(\Gamma_j)\cap \{x_2=\frac{1}{2}\}$ are nonempty). In addition, the estimate on the angle is same as the one in Lemma \ref{flat}, we omit the argument on (b) and only focus on the proof of (a).

For nontrivial cases, the method to define $\Gamma_{j+1}$ is a combination of those in Proposition \ref{step2-L2.2} and Lemma \ref{flat}.
In fact, the difference between here and Proposition \ref{step2-L2.2} lies in that $T(\Gamma_j)$ is no longer a segment here  and thus $T(\Gamma_j)\cap \{x_1=\frac{1}{2}\}$ (or  $T(\Gamma_j)\cap \{x_2=\frac{1}{2}\}$)
may has two or more points. In addition, the `average slope' of the simple curve  can be arbitrary. In contrast, for the situation in
Proposition \ref{step2-L2.2}, the slope of the segment is $\pm 1$ and hence the argument there is much simpler.
%If it occurs, we follow the method of Lemma \ref{flat}, that is,
%we replace $T(\Gamma_j)$ by a simple continuous (still piecewise ${\cal C}^2$-smooth) curve $\widehat{T(\Gamma_j)}$ totally in some small square such that $T(T(\Gamma_j))=T(\widehat{T(\Gamma_j)})$.
%Then we define $\Gamma_{j+1}=\widehat{T(\Gamma_j)}$. If $T(\Gamma_j)\cap \{x_1=\frac{1}{2}\}$ (or  $T(\Gamma_j)\cap \{x_2=\frac{1}{2}\}$) includes at most one point, we will apply the method of Proposition \ref{step2-L2.2} as follows
 %that is, let $\Gamma_{j+1}$ be one of two component of $T(\Gamma_j)$ such that $\bm(\Gamma_{j+1})\ge \frac12 \bm(T(\Gamma_j))$. For the case that there is no intersection point, we simply define $\Gamma_{j+1}=T(\Gamma_j)$

%\vskip 0.2cm
%``Proof of the claim (a)''. It is sufficient to prove that there is $i\in \{k, k+1, k+2\}$ ($0\le k\le j-1$) such that $T(\Gamma_i)\cap \{x_1=\frac{1}{2}\}$ and $T(\Gamma_i)\cap \{x_2=\frac{1}{2}\}$ are empty. In other word, among every three iterations, there is at least one in which no folding occurs. In fact, it implies $\Gamma_{i+1}=T(\Gamma_i)$ and $\bm(\Gamma_{i+1})\geq (\lambda-a)\bm(\Gamma_i)$. On the other hand, for any $k$, from the definition of $\Gamma_k$ we have $\bm(\Gamma_{k+1})\geq \frac{1}{2}\bm(T(\Gamma_k))\geq \frac{\lambda-a}{2}\bm(\Gamma_k)$. Consequently, we obtain  $\bm(\Gamma_{i+3})\geq \frac{(\lambda-a)^3}{2^2}\bm(\Gamma_i)\stackrel{\Delta}{=}\hat{e} \bm(\Gamma_i)$ with $\hat{e}>1$ if $ | \lambda-2 | ,  | a | \ll 1$.
\vskip 0.2cm
%Next we will follow the proof of Proposition \ref{step2-L2.2} to prove the existence of $i$.

\item  $T(\Gamma_j)\cap (\{x_1=\frac{1}{2}\}\cup\{x_2=\frac{1}{2}\})=\emptyset$. Define $\Gamma_{j+1}=T(\Gamma_j)$. Then $\bm(\Gamma_{j+1})\ge (\lambda-O(\eta ))\bm(\Gamma_j)$ with $\lambda-O(\eta )>1$ if $ | \lambda-2 | ,  |\eta | \ll 1$.
%Then if $j+1\ge 3$, we obtain
%$\bm(\Gamma_{j+1})\geq \frac{(\lambda-a)^3}{2^2}\bm(\Gamma_i)\stackrel{\Delta}{=}\hat{e} \bm(\Gamma_i)$ with $\hat{e}>1$ if $ | \lambda-2 | ,  | a | \ll 1$.

\item $T(\Gamma_j)\cap (\{x_1=\frac{1}{2}\}$  has two or more points and $T(\Gamma_j)\cap (\{x_2=\frac{1}{2}\}=\emptyset$ (or vice versa).  We replace $T(\Gamma_j)$ by a simple  (still piecewise ${\cal C}^2$-smooth) curve $\widehat{T(\Gamma_j)}$ totally in some small square (by reflecting $T(\Gamma_j)$ with respect to $x_1=\frac{1}{2}$ and $x_2=\frac{1}{2}$, respectively)  such that $T(T(\Gamma_j))=T(\widehat{T(\Gamma_j)})$.
Then we define $\Gamma_{j+1}=\widehat{T(\Gamma_j)}$ (although $\Gamma_{j+1}$ is not in the image of $\Gamma_j$, for our purpose, $\widehat{T(\Gamma_j)}$ and ${T(\Gamma_j)}$ are equivalent) and the case is similar to that in  Case 2.

%Both $T(\Gamma_j)\cap \{x_1=\frac{1}{2}\}$ and  $ T(\Gamma_j)\cap \{x_2=\frac{1}{2}\}$ are one-point sets. Then
%$T(\Gamma_j)\cap \{x_1=x_2\}\neq \emptyset$, or $T(\Gamma_j)\cap \{x_1+x_2=1\}\neq \emptyset$ which implies %$T^2(\Gamma_{j})\cap \{x_1=x_2\}\neq \emptyset$. This will leads to the conclusion of Lemma \ref{nonlinear-after-delta1} as %above.

\item $ T(\Gamma_j)\cap \{x_2=\frac{1}{2}\})$ is one-point set and  $T(\Gamma_j)\cap \{x_1=\frac{1}{2}\}=\emptyset$ (or vise versa). Let
$\Gamma_{j+1,1}\cup\Gamma_{j+1,2}$ be two components of $T(\Gamma_j)$, i.e., $\Gamma_{j+1,l}$ is in some small square $(l=1,2)$. It can be divided into the following subclasses.
\begin{enumerate}
\item Assume $\max \{\bm(\Gamma_{j+1,1}), \bm(\Gamma_{j+1,2})\}\geq (1+e-O(\eta ))\bm(\Gamma_j)$, $e=\frac{\lambda^2}{\lambda+1}-1$. Without loss of generality, let $\bm(\Gamma_{j+1,1})\ge\bm(\Gamma_{j+1,2})$. Then the proof of  claim (a) is completed by setting $\Gamma_{j+1}=\Gamma_{j+1,1}$ and $\hat{e}=1+e-O(\eta )$.

\item Assume $\max \{\bm(\Gamma_{j+1,1}), \bm(\Gamma_{j+1,2})\}\leq (1+e-O(\eta ))\bm(\Gamma_j)$. By (\ref{sum-length}), we have $\bm(\Gamma_{j+1,l})\geq (\lambda-(1+e)-O(\eta ))\bm(\Gamma_j), \ l=1,2.$ Denote $\widehat{\Gamma}_{j+1,l}=T(\Gamma_{j+1,l}), \ l=1,2.$ We need to consider the following sub-cases:
\begin{enumerate}
\item  $ \widehat{\Gamma}_{j+1,1}\cap (\{x_1=\frac{1}{2}\}\cup\{x_2=\frac{1}{2}\})=\emptyset$  (or $ \widehat{\Gamma}_{j+1,2}\cap (\{x_1=\frac{1}{2}\}\cup\{x_2=\frac{1}{2}\})=\emptyset$). Then it holds that
\begin{align*} \bm(\widehat{\Gamma}_{j+1,1})\geq & (\lambda-O(\eta ))\bm(\Gamma_{j+1,1})\geq  (\lambda-O(\eta ))(\lambda-(1+e)-O(\eta ))\bm(\Gamma_j)\\
\geq &  (1+e-O(\eta ))\bm(\Gamma_j)\geq \hat{e}\bm(\Gamma_j).
\end{align*} Thus it is sufficient to choose $\Gamma_{j+1}=\widehat{\Gamma}_{j+1,1}$.

\item  Both $\widehat{\Gamma}_{j+1,1}\cap(\{x_1=\frac{1}{2}\}\cup\{x_2=\frac{1}{2}\})$ and  $ \widehat{\Gamma}_{j+1,2}\cap(\{x_1=\frac{1}{2}\}\cup\{x_2=\frac{1}{2}\})$ are one-point sets.
 %Note that $\Gamma_{1,1}$ and $\Gamma_{1,2}$ stays in two different small squares. Let $\widetilde{\Gamma}_{1,2}$ be the equivalent curve of $\Gamma_{1,2}$ (i.e. $T(\widetilde{\Gamma}_{1,2})=T(\Gamma_{1,2}))$ such that
%$\widetilde{\Gamma}_{1,2}$ and $\Gamma_{1,1}$ lie in the same small squares.

Note that angles between $\Gamma_{j+1,1}$, ${\Gamma}_{j+1,2}$ and a vertical line are nearly equal to each other at the point $\Gamma_{j+1,1}\cap {\Gamma}_{j+1,2} $, since $\mathrm{r_a}(\Gamma_{j})\ll 1$
if $ | \eta  | \ll 1$ and $T|_{\Gamma_j}$ is close to a linear map. Subsequently, since
$\mathrm{r_a}(\Gamma_{j+1,1}),\  \mathrm{r_a}({\Gamma}_{j+1,2})\ll 1$ if $ | \eta  | \ll 1$, we have that angles between the tangent line at any point of $\Gamma_{j+1,1}$ or ${\Gamma}_{j+1,2} $  and a vertical line are nearly two constants which are nearly equal to each other.
In other words, $\Gamma_{j+1,1}$ and ${\Gamma}_{j+1,2}$ almost lie in two lines symmetric corresponding to a vertical line. The case for $\widehat{\Gamma}_{j+1,1}$ and $\widehat{\Gamma}_{j+1,2}$ is similar if $ | \eta  | \ll 1$. Without loss of generality, it is sufficient to consider the following two subcases.
\begin{enumerate}
 \item  Both $T(\Gamma_{j+1,l})\cap\{x_1=\frac{1}{2}\}$, $l=1,2$ are nontrivial one-point set.
 For this case, we have  $\widehat{\Gamma}_{j+1,l}(l=1,2)$ almost lie in two lines which are very close to the vertical line $x_1=\frac{1}{2}$.
It implies that  $T(\widehat{\Gamma}_{j+1,l})$ almost coincides with the boundary $x_1=1-\tau_1$.

If $T(\widehat{\Gamma}_{j+1,1})\cap\{x_2=\frac{1}{2}\}=\emptyset$, then the argument is completed by setting $\Gamma_{j+1}=T(\widehat{\Gamma}_{j+1,1})$.

Consider the case $T(\widehat{\Gamma}_{j+1,1})\cap\{x_2=\frac{1}{2}\}\not=\emptyset$. Denote components of  $T(\widehat{\Gamma}_{j+1,1})$ by $\widetilde{\Gamma}_{j+1,l}, l=1, 2$. Obviously,  $T(\widetilde{\Gamma}_{j+1,l})$ is close to $x_1=\tau_2\ll 1$ and thus has no intersection with $x_1=\frac12$. Hence we only need to consider the intersection of it with $x_2=\frac12$.
It thus can be reduced to the following subcase.
\vskip 0.2cm
%%%%%%%%%%%%%%%%%%%%%%%%%%%%%%%%%%%%%%%%%%%%%%%%%%%%%%%%%%%%%
\item  $T(\Gamma_j)\cap\{x_2=\frac{1}{2}\}$ and $(\widehat{\Gamma}_{j+1,1}=T(\Gamma_{j+1,1}))\cap\{x_2=\frac{1}{2}\}$ are nontrivial  one-point set, which are denoted by $\bp _j$ and $\bp _{j+1}$, respectively. Obviously, $\bp _j\in \Gamma_{j+1,1}$. From the fact that $T$ is a perturbation of uncoupled tent map (with a slope $ s =2-s_0$ satisfying $1\gg s_0\gg\eta  \geq 0$), we have that
$$ | x_2(T(\bp _{j}))-1 | =O(\max\{ s_0,c_1,\eta \})\equiv O(\tau_0).$$
Obviously $x_2(\bp _{j+1})=\frac{1}{2}$, hence $x_2(T(\bp _{j+1}))=1-O(\tau_0)$. Similarly, $x_2(T^2(\mathbf{p}_{j}))=O(\tau_0)$.
Since $T^2(\mathbf{p}_{j}), T(\bp _{j+1})\in T(\widehat{\Gamma}_{j+1,1}))=\widetilde{\Gamma}_{j+1,1}$ and $\mathrm{r_a}(\widehat{\Gamma}_{j+1,1})=O(\eta )$, roughly speaking, $\widetilde{\Gamma}_{j+1,1}$ is nearly a vertical segment from the bottom to the top.
Define $\Gamma^1_{j+1}$ be the component of  $\widetilde{\Gamma}_{j+1,1}$  such that $\max_{\bp \in\Gamma^1_{j+1}}x_2(\bp )=\frac{1}{2}$, $\min_{\bp \in\Gamma^1_{j+1}}x_2(\bp )=O(\tau_0)$, $\mathrm{r_a}(\Gamma^1_{j+1})=O(\eta )$.

If $T(\Gamma^1_{j+1})\cap(\{x_1=x_2\}\cup\{x_1+x_2=1\})\not=\emptyset$, then the proof is completed. Thus assume the intersection set is empty.
 Then it is necessary that
$\max_{\bp \in T(\Gamma^1_{j+1})}x_1(\bp )=O(\eta+\tau_0)$.
% In fact, let $\bp \in\{x_1+x_2=1\}$ and $\mathbf{q}\in\{x_1=x_2\}$ with $x_1(\bp ), x_1(\mathbf{q})\in[0,\frac{1}{2}]$. If $x_1(\bp )\gg \epsilon+\tau_0$, then $\epsilon+\tau_0\ll 1-\mathrm{dist} (\bp ,\mathbf{q})$,
Otherwise, $\max_{T(\Gamma_{j+1}^{i_0})}x_1(\bp)\gg \eta+\tau_0$. Note that $T(\Gamma_{j+1}^{i_0})$ is nearly a segment. Let $\bp_1, \bp_2$ be two end points of $T(\Gamma_{j+1}^{i_0})$ such that $x_2(\bp_1)\ge 1-O(\tau_0)-O(\eta)$ and $x_2(\bp_2)=O(\tau_0)$. Again from the fact that $T(\Gamma_{j+1}^{i_0})$ is nearly a segment, we have either  $x_1(\bp_1) \gg \eta+\tau_0$ or $x_1(\bp_2) \gg \eta+\tau_0$. For the former case, we have  $x_1(\bp_1)+x_2(\bp_1)-1\gg \eta+\tau_0-\tau_0-\eta =0$, while $x_1(\bp_2)+x_2(\bp_2)\le 1/2+O(\tau)\le 1$, which implies the intersection between $T(\Gamma_{j+1}^{i_0})$ and $x_1+x_2=1$ is nonempty. For the latter case, we have $x_2(\bp_2)<x_1(\bp_2)$, while $x_2(\bp_1)>1/2>x_1(\bp_1)$, which implies that   $T(\Gamma_{j+1}^{i_0})$ and $x_1=x_2 $ is nonempty.
 which makes the assumption on the empty intersection impossible.
%Thus $T(\Gamma^1_{j+1})$ is nearly a vertical segment from the bottom to the top.
%

Define $\Gamma^2_{j+1}$ as before such that $\max_{\bp \in\Gamma^2_{j+1}} x_2(\bp )=\frac{1}{2}$, $\min_{\bp \in\Gamma^2_{j+1}} x_2(\bp )=O(\tau_0), $ $\mathrm{r_a}(\Gamma^2_{j+1})=O(\eta )$. From (\ref{*}), it holds that there is an $i_0\leq [\log_{(s-\eta)(1-c)} \frac{1}{100}/\tau_2]+1$, such that
$\max_{\bp \in\Gamma^2_{j+1}} x_1(T^{i_0}(\bp ))\in [\frac{1}{100} ,\frac{1}{10}]$.
In fact, since $x_1(\bp )\geq  \tau_2  $, from the argument above, we have $x_1(T^i (\bp ))\geq (s-\eta)(1-c)x_1(T^{i-1}(\bp ))$ only if $x_1(T^{i-1}(\bp ))\leq \frac{1}{100}$, which justifies the definition of $i_0$. In a word, we obtain a simple  curve $\Gamma^{i_0}_{j+1}$ satisfying $\max_{\bp \in\Gamma^{i_0}_{j+1}} x_2(T(\bp ))=1-O(\tau_0)$,
$\min_{\bp \in\Gamma^{i_0}_{j+1}} x_2(T(\bp ))=O(\tau_0)$,
$ \max_{\bp \in\Gamma^{i_0}_{j+1}} x_1(T(\bp ))\in \left[ \frac{1}{100},\frac{1}{10}\right]$
%{\color{blue} $\min_{\bp \in\Gamma^{i_0}_{j+1}} x_1(T(\bp ))\leq \frac{1}{10}$}
and
$\mathrm{r_a}(T(\Gamma^{i_0}_{j+1}))=O(\eta )$.
 By the same argument above, we obtain $T(\Gamma^{i_0}_{j+1})\cap (\{x_1=x_2\}\cup\{x_1+x_2=1\})\neq\emptyset$. We can end the proof of Lemma \ref{nonlinear-after-delta1} similarly as before.

%%%%%%%%%%%%%%%%%%%%%%%%%%%%%%%%%%%%%%%%%%%%%%%%%%%%%%%%%%%%%
 %and thus $T(\widehat{\Gamma}_{j+1,l})\cap (\{x_1=\frac{1}{2}\}\cup\{x_2=\frac{1}{2}\})=\emptyset$. This ends the proof by choosing $\Gamma_{j+1}=T(\widehat{\Gamma}_{j+1,1})$ and $\hat{e}=\frac{(\lambda-a)^2}{2}(\lambda-(1+e)-O(a))>1$ for small $e$ and $a$, since $\bm(T(\widehat{\Gamma}_{j+1,1}))\ge \frac{(\lambda-a)}{2}\bm(\widehat{\Gamma}_{j+1,1})\ge \frac{(\lambda-a)^2}{2}\bm(\Gamma_{j+1,1})\ge \frac{(\lambda-a)^2}{2}(\lambda-(1+e)-O(a))\bm(\Gamma_j).$

\end{enumerate}

\item  $\widehat{\Gamma}_{j+1,1}\cap(\{x_1=\frac{1}{2}\}$ (or $\widehat{\Gamma}_{j+1,1}\cap(\{x_2=\frac{1}{2}\}$) includes two or more points, or $\widehat{\Gamma}_{j+1,2}\cap(\{x_1=\frac{1}{2}\}$ (or $\widehat{\Gamma}_{j+1,2}\cap(\{x_2=\frac{1}{2}\}$) includes two or more points. It can be reduced to Case $4$-$(b)$-${\rm i}$ by the same argument as in Case 2, we omit it here.\hfill\qedbox\end{enumerate}\end{enumerate}\end{enumerate}
\end{proof}
\vskip 0.3cm
       \noindent{\it Proof of (\ref{order-part}) in Theorem \ref{nonlinear}}\qquad
       From Lemma \ref{flat} and \ref{nonlinear-after-delta1}, we obtained that there exists $c_6>0$ such that for any segment $\Gamma_0$ in some small square
       with a slope $1$, there exists disjoint segments $\Gamma_1, \Gamma_2, \cdots, $ such that for each $i$ there exists $l(i)$ such that $T^{l(i)}(\Gamma_i)\subset G_{\epsilon}$ and
       $\sum_i\bm(\Gamma_i)\ge c_6\bm(\Gamma_0)$. From the arbitrariness of $\Gamma_0$ and Fubini's Theorem, we obtain that for almost every initial point $\bp \in [0, 1]^2$,
       there exists $l(\bp )$ such that $T^{l(\bp )}(\bp )\in G_{\epsilon}$. This completes the proof for (\ref{order-part}).\hfill \qedbox

\vskip 0.2cm
    The disordered part can be obtained by a series of lemmas. First we prove a partial result as follows.

   \begin{Lemma}\label{5.4} There exists a subset $B_0$ of $G_{\gamma_0}$ with a positive measure such that for each point $\bp \in B_0$, there exists a finite time $l=l(\bp )$ such that
    $T^{l}(\bp )\in B_{\gamma_0}$.
    \end{Lemma}
    \begin{Proof}
    Let $G_1=\{\bp \in G_{\gamma_0} |  x_1,x_2\le 1/2\}$, $G_2=\{\bp \in G_{\gamma_0} |  x_1,x_2\ge 1/2\}$, $G_3=\{\bp \in G_{\gamma_0} |  x_1\le 1/2, x_2\ge 1/2\}$, $G_4=\{\bp \in G_{\gamma_0} |   x_1\ge 1/2, x_2\le 1/2\}$ and define
     $\widetilde{G}_{\gamma_0}=G_3\cup G_4$.  Let $\widetilde{G}_{-1}=T^{-1}(\widetilde{G}_{\gamma_0})\cap G_{\gamma_0}$, $\widetilde{G}_{-2}=T^{-1}(\widetilde{G}_{-1})\cap G_{\gamma_0}$, $\cdots, \widetilde{G}_{-(k+1)}=T^{-1}(\widetilde{G}_{-k})\cap G_{\gamma_0}$. Denote $G_0=\cup_{k=0}^{\infty}\widetilde{G}_{-k}$.

 Now, we define $B_0\equiv G_{\gamma_0}\backslash  G_0$.  Then for each point $\bp \in B_0$, there exists a finite time $l=l(\bp )$ such that
    $T^{l}(\bp )\in B_{\gamma_0}$.  In fact, from the definition of the set $G_0$ and the map $T$, for each point $\bp \in B_0$ satisfying $T^{i-1}(\bp )\in G_{\gamma_0}$, it holds that $ | x_1(i)-x_2(i) | \ge (1-2c)(2-\eta)   | x_1(i-1)-x_2(i-1) | $, where $(x_1(j),x_2(j))=T^{j}(\bp )$ for $j=i-1, i$. Hence, the fact that $(1-2c)(2-\eta)>1$ with $c,\eta$ small implies that there exists some $l$ such that $T^{l}(\bp )$ is out of the region $G_{\gamma_0}$, i.e. enter into $B_{\gamma_0}$.

    Thus it is sufficient to prove that
    $\bm(G_0)<\bm(G_{\gamma_0})$. Since the diameter of $\widetilde{G}_{\gamma_0}$ is small for small $\gamma_0$, from the expansibility of $T$, we obtain that each component $S$ of $T^{-i}(\widetilde{G}_{\gamma_0})$ possesses a small diameter for $i>0$. Thus,
     we have that
    \begin{align*}
     {\rm Case {\rm (i)}.}\quad & \#^C( T^{-1}(S)\cap G_{\gamma_0})=2 \quad {\rm if} \ \mathrm{dist} (S, (1, 1))>3\gamma_0;\\
      {\rm Case {\rm (ii)}.}\quad & \#^C( T^{-1}(S)\cap G_{\gamma_0})\le 4 \quad{\rm if} \ \mathrm{dist} (S, (1, 1))\le 3\gamma_0,\end{align*}
          where $\#^C(S)$ denotes the number of components for the set $S$.% and ${\mathrm{dist} ( )}$ represents the distance between a point and a set.

     On the other hand, the expansibility of $T$ implies that the measure of each component of $T^{-1}(S)$ is less than $((1-2c)(2-\eta)^2)^{-1}  \bm(S)$. Thus for case {\rm (i)} we have that $\bm(T^{-1}(S))\le 2((1-2c)(2-\eta)^2)^{-1}  \bm(S)\equiv  (2\hat{\lambda})^{-1}  \bm(S)$ and for case {\rm (ii)} we have that $\bm(T^{-1}(S))\le 4((1-2c)(2-\eta)^2)^{-1}  \bm(S)\equiv  \hat{\lambda}^{-1}  \bm(S)$ with $1-O(c+\eta)\le\hat{\lambda}\le 1$.

      Obviously, since $\gamma_0$ is small, we have that
     among $\{i,i+1,\cdots, i+10\}$, there is at most one number $j$ such that $\mathrm{dist} (T^{-j}(S), (1, 1))\le 3\gamma_0$.
     Thus it is not difficult to see that for $k=10l+j$ with $1\le j < 10$, it holds that
     $
     \bm(\widetilde{G}_{-k})\le b_k \bm(\widetilde{G}_{\gamma_0})
     $
     with $b_k = \left((2\hat{\lambda})^{-9}\hat{\lambda}^{-1}\right)^{l} (2\hat{\lambda})^{-(j-1)}  \hat{\lambda}^{-1}.$ Let $b_0=1$. For small $c$ and $\eta$ we can easily see that $\sum_{k=0}^{\infty}b_k\le 4.$  Hence we have
     \beq\label{inequality-lemma5.4}\bm(\cup_{k=0}^{\infty}\widetilde{G}_{-k})\le \sum_{k=0}^{\infty}b_k \bm(\widetilde{G}_{\gamma_0})\le 4\bm(\widetilde{G}_{\gamma_0})\le 4\gamma_0  \bm(G_{\gamma_0})<\bm(G_{\gamma_0}).
     \eeq
     This completes the proof of the lemma.\hfill\qedbox
    \end{Proof}

The disordered part (\ref{disorder-part}) can be easily obtained from the following corollary.
\begin{Corollary}
For almost each point $\bp $ in $G_{\gamma_0}$, there exists a finite time $l=l(\bp )$ such that
    $T^{l}(\bp )\in B_{\gamma_0}\cup B_0$, where $B_0$ is defined as in Lemma \ref{5.4}. \end{Corollary}
    \begin{Proof}
    Assume the conclusion is not true. Then there exists a set $S\subset G_{\gamma_0}$ with a positive measure such that for each $i$ it holds that $T^i(S)\cap (B_{\gamma_0}\cup B_0)=\emptyset.$ We will prove that there exists a subset $S_0\subset S$ with a positive measure and $l=l(S_0)$ such that $T^l(S_0 ) \subset B_{\gamma_0} \cup B_0 $. From the contradiction, we end the proof.

    For this purpose, we claim that there exists a subset $\widetilde{S_0}\subset S$ and $l\in\NN$ such that $\bm(T^{l}(\widetilde{S_0}))\ge \delta_1 \bm(G_{\gamma_0})$.

    Then the Corollary can be obtained from the claim. In fact, since $T^{l}(\widetilde{S_0})\subset T^{l}(S)\subset G_{\gamma_0}$ and $B_0=G_{\gamma_0}\backslash G_0$, we have $T^{l}(\widetilde{S_0})\backslash B_0\subset G_0$, where $G_0$ is defined as in the proof of Lemma \ref{5.4}. Note that $T$ is a diffeomorphism on each small square, thus the image of each measurable set under $T$ is still measurable.
    Thus $$\bm(T^{l}(\widetilde{S_0})\cap B_0)= \bm(T^{l}(\widetilde{S_0}))-\bm(T^{l}(\widetilde{S_0})\backslash B_0)\ge \bm(T^{l}(\widetilde{S_0}))-\bm(G_0).$$ Then from the claim and (\ref{inequality-lemma5.4}), we have
    $$
    \frac{\bm(T^{l}(\widetilde{S_0})\cap B_0)}{\bm(T^{l}(\widetilde{S_0}))}\ge \frac{\bm(T^{l}(\widetilde{S_0}))-\bm(G_0)}{\bm(T^{l}(\widetilde{S_0}))}\ge 1-\frac{4\gamma_0}{\delta_1}\ge 1/2, \quad {\rm since\ \gamma_0\ll \delta_1}.
    $$
    Define $S_0\subset \widetilde{S_0}$ such that $T^{l}(S_0)=T^{l}(\widetilde{S_0})\cap B_0\subset B_0$. Then $T^{l}(S_0)$ has a positive measure, which implies the measure of $S_0$ is positive. Thus the conclusion is obtained.

    Next we prove the existence of $\widetilde{S_0}$.
   From the definition of $G_i,\ i=1, 2, 3, 4$ and the number of components of a set in $[0, 1]^2$, we have that
    for any subset $S$ of $G_3$ or $G_4$, $T^j(S)$ has only one component for each $j=1, 2, \cdots, 10$. Recall that for any subset $S$ of $G_1$ or $G_2$, $T(S)$ has at most 4 components (see Figure \ref{fig:5.1}).

    Without loss of generality, we assume that $S\subset G_1$. Then $T(S)$ has at most 4 components among which 2 components, say $S_{1}$ and $S_{2}$ lie in $G_1$ and $G_2$, respectively, and $S_{3}$ and $S_{4}$ lie in $G_3$ and $G_4$. Subsequently, $T(S_{1})$ or $T(S_{2})$ have at most $4$ components denoted by $S_{i,j}$ $i=1,2,\ j=1, 2, 3, 4$ in a similar way, while $T(S_{3})$ or $T(S_{4})$ has only one components denoted by $S_{i,1}$, $i=3, 4$. Hence, all $T(S_{1,i}),\ i=1, 2, 3, 4$ totally have $10$ components. Moreover, among them there are six components, that is $S_{i,1}$, $i=3, 4$ together with $S_{i,j}$ $i=1,2,\ j=3, 4$, satisfy that the image of each of them has only one component.

\begin{figure}[H]
\centering
\includegraphics[height = 10 cm, width = 14 cm]{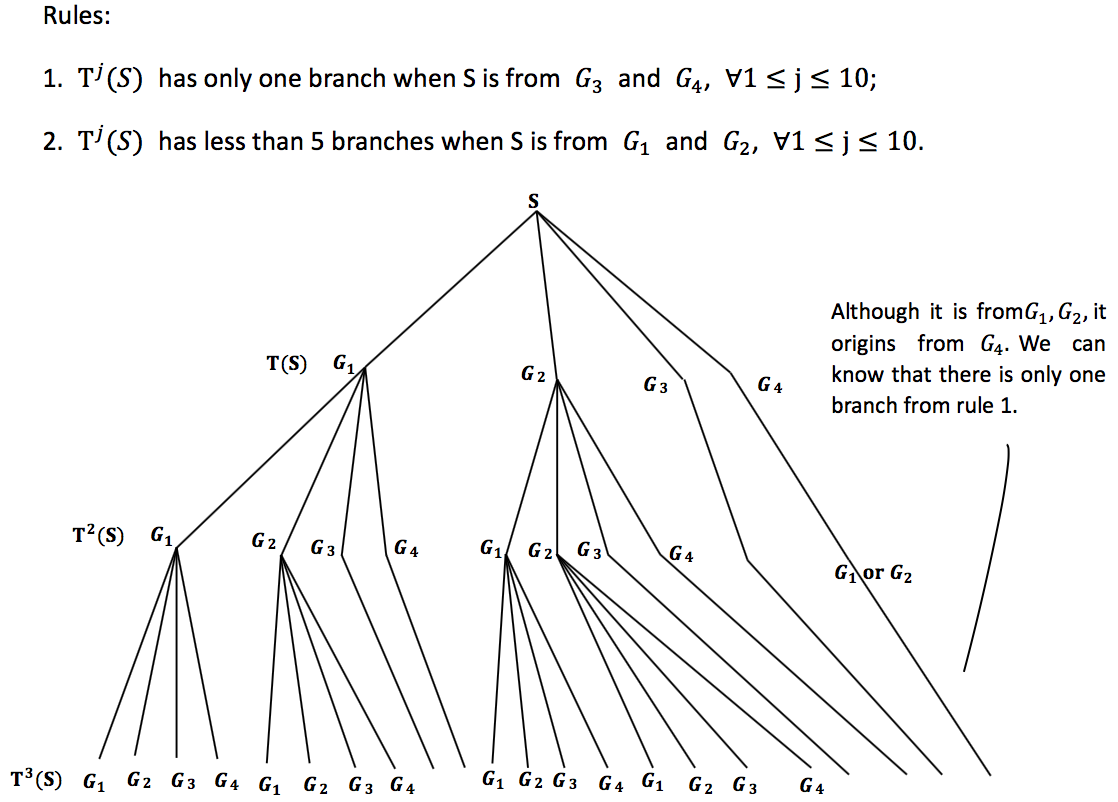}
\caption{Evolution on components of $T^j(S),\ j=1,2,3.$ }
\label{fig:5.1} % ÓÃÓÚœ»²æÒýÓÃ
\end{figure}

    By induction, we can prove that for $i\le 10$, it holds that the sum of all components for $T^i(S)$ is $2^{i+1}+2^i-2$.
 In particular, the sum of all components for $T^3(S)$ is 22.

    On the other hand, $\bm(T^3(S))\ge (4(1-2c-2\eta))^3  \bm(S)$. When $c,\eta$ are small, we have
     that $(1-2c)(2-\eta)^2>(22)^{1/3}$. Applying  Corollary \ref{Corollary} by setting $a=22, m_0=3, E_-(c)=(1-2c)(2-\eta)^2, E_+(c)=(1-2c)(2+\eta)^2,\delta_1=2^{-16},\gamma_0=2^{-20}$ and
     $D=G_{\gamma_0}$,  we obtain the existence of $\widetilde{S_0}$. The proof is complete.\hfill\qedbox
    \end{Proof}

    \section{The higher dimensional case}
    In this section, we will prove Theorem \ref{multi-node-linear} for coupled tent map lattices with multi-node.

    The proof for the ordered part in the multi-node case is quite different from for two-nodes case. The observation is as follows.

    Recall that the phase space $[0, 1]^m=\cup D_J$, where  $m$ is the dimension of the phase space and $D_J$ are $2^m$ small hypercubes in the phase space divided by the planes $x_i=1/2,\ 1\le i\le m$. For each convex $\Omega$ in some $D_{J_0}$, we have that $\bm(T(\Omega))\approx 2^m  \bm(\Omega)$. Clearly $T(\Omega)$ will either has an intersection with each of $2^m$ small hypercubes $D_J$ simultaneously, or there exists at least one hypercube which has no intersection point with $T(\Omega)$. Note that all $T(\Omega)\cap D_J$ are still convex. Once the former case occurs, from the convexity we can prove that $T(\Omega)$ has an intersection with the diagonal $D_{\mathrm{syn}}$, which again by convexity implies the existence of a set of `good' points occupying a fixed ratio in $\Omega$. Otherwise, suppose the latter case occurs in each iteration step. Then in each iteration step, averagely it holds that $\bm(T(\Omega)\cap D_J)\ge c  \bm(\Omega)$ with $c\approx 2^m/(2^m-1)>1$ for each $J$. Consequently, the measure for most components of $T^k(\Omega)$  will keep increasing until it is of constant order as $k$ increases. Thus we also obtain the existence of a set of `good' points occupying a fixed ratio in $\Omega$ by convexity.

    To prove the ordered part (\ref{order-part}), we first have the following result.
    \begin{Lemma}\label{center-point}
    If $\Omega$ is a convex set in one of the $2^m$ small hypercubes $D_J$  and $T(\Omega)$ has an intersection with each small hypercube simultaneously. Then the center point $x_1=x_2=\cdots =x_m=1/2$ lies in $T(\Omega)$. Moreover, there exists a fixed number $c_0>0$ such that for any $\epsilon>0$, we have that $\bm(T(\Omega)\cap G_{\epsilon})/\bm(T(\Omega))\ge c_0  \epsilon^m$, where $$G_{\epsilon}=\{\bp\in [0,1]^m |  \mathrm{dist}(\bp,D_{\mathrm{syn}})\le \epsilon\}.$$
   \end{Lemma}
    \begin{Proof}
    We prove the first conclusion by induction.   For $m=2$, let $$D_{ij}=\{(x_1,x_2) \in [0,1]^2 |  \frac12 (i-1)\le x_1\le \frac12 i,\
    \frac12 (j-1)\le x_2\le \frac12 j\},$$ $i, j=1, 2,$ be all the small hypercubes. From the condition, we have that for each pair $(i,j)$, there exists a point $\bp _{ij}\in D_{ij}\cap T(\Omega)$. Then from the convexity, we have that the convex hull determined by these points is a subset of $T(\Omega)$ and the point $(1/2, 1/2)$ is in it. Thus the first conclusion is proved.

    Assume the first conclusion holds true for $k=2, \cdots, m-1$. For the case $k=m$, consider $\widehat{D}_{1/2}=\{(x_1,\cdots, x_m)\in [0,1]^m |  x_m=1/2\}$. Obviously it consists of $2^{m-1}$ small $(m-1)$-dimensional hypercubes determined by the planes $x_i=1/2,\ i=1, \cdots, m-1$.  We have that $\widehat{D}_{1/2}\cap T( \Omega)$ is nonempty and convex by the convexity of $\Omega$, since in $T(\Omega)$ there exist both points with $x_m<1/2$ and the ones with $x_m>1/2$. Furthermore, the condition implies that the intersection between $T(\Omega)$ and each small hypercubes of $\widehat{D}_{1/2}$ is nonempty. Thus applying inductive assumption for $m-1$ on $\widehat{D}_{1/2}\cap T( \Omega)$ and $\widehat{D}_{1/2}$, we have that the point $(1/2, \cdots, 1/2)\in \widehat{D}_{1/2}\cap T( \Omega)$, which leads to the first conclusion.

    For the second conclusion, let $S_{\epsilon}$ be the cylinder $\{\bp\in [0,1]^m |  \mathrm{dist}(\bp,D_{\mathrm{syn}})= \epsilon\}$ whose axis is the diagonal of the phase space.

   % Let $T(\Omega)=T(\Omega)_{in}\cup T(\Omega)_{out}$ with $T(\Omega)_{in}=T(\Omega)\cap G_{\epsilon}$ and
   % $T(\Omega)_{out}=T(\Omega)\backslash G_{\epsilon}$. Obviously, to prove the second conclusion, it is sufficient to prove
    %that the inequality holds true for , that is, $\bm(T(\Omega)_{out}\cap G_{\epsilon})/\bm(T(\Omega))\ge c_0\cdot \epsilon^m$

    Define $\widehat{S}_{\epsilon}=S_{\epsilon}\cap T(\Omega)$ and let $\Gamma$ be the point set of the union of all lines connecting $\widehat{S}_{\epsilon}$ and the point $\bp _0=(1/2, \cdots, 1/2)$. From the first conclusion it holds that $\bp _0\in T(\Omega)$. Hence, if $\widehat{S}_{\epsilon}$ is empty, then the convexity of $T(\Omega)$ leads that $T(\Omega)\subset G_{\epsilon}$ and the proof is complete. Thus we assume that $\widehat{S}_{\epsilon}$ is nonempty. Obviously, $T(\Omega)\backslash (\Gamma\cap T(\Omega))\subset T(\Omega)\cap G_{\epsilon}$. Thus for our purpose, we only need to analyze $\Gamma\cap T(\Omega)$.
    Again by the convexity we have that $\Omega_{\epsilon}\equiv  \Gamma\cap G_{\epsilon}\subset T(\Omega)$, which further implies \begin{equation}\label{omega-epsilon}\Omega_{\epsilon}= (\Gamma\cap T(\Omega))\cap G_{\epsilon}.\end{equation} To prove the second conclusion, it is sufficient to estimate $\bm(\Omega_{\epsilon})/\bm(\Gamma\cap T(\Omega))$. Note that $\Gamma\cap T(\Omega) \subset \Gamma\cap [0, 1]^m$.
    %For this purpose, we claim that $T(\Omega)\backslash G_{\epsilon} \subset \widetilde{\Omega}_{\epsilon}=\Gamma\cap [0, 1]^m$. In fact, consider any point $P\in T(\Omega)\backslash G_{\epsilon}$.  Let $Q$ be the intersection point between $S_{\epsilon}$ and the line connecting $P$ and $\bp _0$. Then by the definition, $Q\in \widetilde{S}_{\epsilon}$, which implies that the segment $\bp _0P\subset \widetilde{\Omega}_{\epsilon}$. This proves the claim.
    For each line $L\in \Gamma$, we can easily see that the length $L\cap [0, 1]^m$ is less than $\sqrt{m}$, while the length of $L\cap  G_{\epsilon}$ is larger than $\epsilon$. For the definition of $\Omega_{\epsilon}$ and (\ref{omega-epsilon}), we thus have that
    $$\bm((\Gamma\cap T(\Omega))\cap G_{\epsilon}) = \bm(\Omega_{\epsilon})\ge (\frac{\epsilon}{\sqrt{m}})^m  \bm(\Gamma\cap [0, 1]^m)\ge (\frac{\epsilon}{\sqrt{m}})^m  \bm(\Gamma\cap T(\Omega)).$$ Hence we complete the proof of the lemma by setting $c_0=m^{-m/2}$.\hfill\qedbox
    \end{Proof}
    \vskip 0.4cm
  \noindent  Applying Lemma \ref{center-point}, we obtain the following result.
\begin{Proposition}\label{firststep-higher-dim} There exists a constant $c_1>0$ such that for any convex set $\Omega$ with a volume less than $10^{-m}$, there exist disjoint convexities $\Omega_i\subset \Omega$, $i=1, 2,\cdots,$ and a set $\Omega_0\subset \Omega$ which is a union of finite convexities such that {\rm(i)} for any $i\ge 1$, there exists a $l(i)$ such that $T^{l(i)}(\Omega_i)$ is a component of $T^{l(i)}(\Omega)$ and $\bm(T^{l(i)}(\Omega_i))\ge \delta_1$; {\rm(ii)} for any point $\bp \in \Omega_0$, there exists an $l(\bp )$ such that $T^{l(\bp )}(\bp )\in G_{\epsilon}$;
{\rm(iii)}$\bm(\cup_{i\ge 0} \Omega_i)\ge c_1  \bm(\Omega).$
\end{Proposition}
\begin{Proof}
From Lemma \ref{center-point}, without loss of generality we assume that there exists at least one small hypercube which has no intersection point with $T(\Omega)$ for any convex $\Omega$ in some small hypercube. Thus $T(\Omega)$ has at most $2^m-1$ components, denoted by $\Omega_i,\ i=1,\ \cdots,\ 2^m-1$ (some of them may be empty).

On the other hand,
 $|\mathrm{det}(JT(\bp))|=2^m|\mathrm{det}(I+cA)|\equiv 2^m(1-F(A,c))$, where $F(A,c)$
depends on $A^\top=A\in \RR^{m\times m}$ with $A\mathbf{e}=0$ and the coupling coefficient $c$ satisfying $F(A, c)\rightarrow 0$ as $c\rightarrow 0$.
Then we have
$$|\mathrm{det}(JT(\bp))|>2^m-1,\quad {\rm for\ small}\ c.$$

Thus from Iteration Lemma \ref{iteration-L} and Corollary \ref{Corollary} with $m_0=1$, $a=2^m-1$ and $E_-(c)=2^m(1-F(A,c))$,  we complete the proof of this proposition.\hfill\qedbox
\end{Proof}

The ordered part (\ref{order-part}) for the case of multi-node can be reduced to the following
result.
\begin{Proposition} \label{step2-higher-dim}
For any $\epsilon>0$, there exists a fixed number $c_2>0$ depending on $\epsilon$ such that  for any convex set $\Omega_0$ in some small hypercube with $\bm(\Omega_0)\ge\delta_1$, there exist disjoint convexities $\Omega_i\subset \Omega_0$, $i=1, 2,\cdots,$ such that {\rm(i)} for any $i\ge 1$, there exists an $l(i)$ such that $T^{l(i)}(\Omega_i)\subset G_{\epsilon}$ and
{\rm(ii)}$\bm(\cup_{i\ge 1} \Omega_i)\ge c_2  \bm(\Omega_0).$
\end{Proposition}
\begin{Proof}
If $T(\Omega_0)\cap D_J\not=\emptyset$ for each small hypercube $D_J$ of $D$, then the conclusion follows from Lemma \ref{center-point}.

 Otherwise, there exist at most $2^m-1$ small hypercubes of $[0, 1]^m$ such that $T(\Omega_0)\cap D_J\not=\emptyset$. Recall that $\bm(T(\Omega_0))\ge 2^m  (1-F(A,c))  \bm(\Omega_0)$. Hence there exists some $D_J$ such that the convex set $\Omega_1=T(\Omega_0)\cap D_J$ possesses a volume larger than $$(2^m/(2^m-1))(1-F(A,c))  \bm(\Omega_0).$$ Similarly, assume that there exists at most $2^m-1$ small hypercubes of $[0, 1]^m$ such that $T(\Omega_1)\cap D_J\not=\emptyset$. Then
 %we can obtain a convex $\Omega_2=T(\Omega_1)\cap D_J$ for some $D_J$ such that $\bm(\Omega_2)\ge 2(1-f(n,c))  \bm(\Omega_1)\ge (2(1-f(n,c)))^2  \bm(\Omega_0)$.
 repeating the above argument, we obtain a convex set $\Omega_{i+1}=T(\Omega_i)\cap D_J$ for some $D_J$ such that $$\bm(\Omega_{i+1})\ge 2^m/(2^m-1)  (1-F(A,c))  \bm(\Omega_i)\ge (2^m/(2^m-1)  (1-F(A,c)))^{i+1}  \bm(\Omega_0)$$ for any $i$. Since $F(A,c)\rightarrow 0$ and $\bm(\Omega_0)\ge \delta_1$, we have that $$(2^m/(2^m-1) (1-F(A,c)))^{j}  \bm(\Omega_0)>1$$ for small $c$ and some fixed $j=j(\delta_1)$. Hence there exists some $i<j$ such that $T(\Omega_i)\cap D_J\not=\emptyset$ for small hypercubes $D_J$ of $[0, 1]^m$, which leads to the conclusion with $$c_2>c_0  \bm(G_{\epsilon})  2^{-mj}\ge c_0  \epsilon^{m-1}  2^{-mj},$$ where $c_0$ is defined in Lemma \ref{center-point}.\hfill\qedbox
\end{Proof}

Similar to the proof of Theorem \ref{liminf}, we can easily obtain (\ref{order-part}) from Proposition \ref{step2-higher-dim}. Thus we omit the details.
\vskip 0.2cm
The disordered part for the multi-node case is similar to the two-node case. Let $\gamma_0=\min\{2^{-20}, 2^{-m-4}\}$.
%The observation is as follows.
%For small $\gamma_0$, we have that any point $P\in \cap G_{\gamma_0}\backslash \widetilde{G}_{\gamma_0}$ possesses only two pre-images which are in $D_{J_1}$ or $D_{J_2}$. is close to the center point $(1/2,\cdots, 1/2)$ (a small triangle for 2-dim case), which implies that . Thus $T^2(D_J\cap G_{\gamma_0})$ will be divided into at most $2^m(2^m-1)$  components (possibly can be refined).
First we prove that
\begin{Lemma}\label{6.2} There exists a subset $B_0$ of $G_{\gamma_0}$ with a positive measure such that for each point $\bp \in B_0$, there exists a finite time $l=l(\bp )$ such that
    $T^{l}(\bp )\in B_{\gamma_0}$.
    \end{Lemma}
    \begin{Proof}
   Let $$D_{J_1}=\{(x_1,\cdots, x_n)\in D |  0\le x_i\le 1/2,\ {\rm for\ all}\ i\},$$
   $$D_{J_2}=\{(x_1,\cdots, x_n)\in D  |  1/2\le x_i\le 1,\ {\rm for\ all}\ i\}.$$ Define $\widetilde{G}_{\gamma_0}=G_{\gamma_0}\backslash (D_{J_1}\cup D_{J_2})$.
   %Obviously $\widetilde{G}_{\gamma_0}$ is a small neighborhood of the center point $\bp _0$ for small $\gamma_0$.
   Let $$\widetilde{G}_{-1}=T^{-1}(\widetilde{G}_{\gamma_0})\cap G_{\gamma_0},\ \widetilde{G}_{-2}=T^{-1}(\widetilde{G}_{-1})\cap G_{\gamma_0},\ \cdots,\  \widetilde{G}_{-(l+1)}=T^{-1}(\widetilde{G}_{-l})\cap G_{\gamma_0}.$$ Denote $G_0=\cup_{l=0}^{\infty}\widetilde{G}_{-l}$.  We claim that for each point $\bp \in B_0=G_{\gamma_0}\backslash G_0$, there exists some $l$ such that $T^{l}(\bp )$ is out of the region $G_{\gamma_0}$, i.e. enters into $B_{\gamma_0}$. Thus for our purpose it is sufficient to prove that
    $\bm(G_0)<\bm(G_{\gamma_0})$.
From the definition of the set $G_0$ and the map $T$,  for each point $\bp \in B_0$, it holds that $T^i(\bp )\not\in \widetilde{G}_{\gamma_0}$ for any $i$.  By the expansivity of $T$, we need to prove that if $T^{i-1}(\bp )\in G_{\gamma_0}$, then $\mathrm{dist} (T^i(\bp ), D_{\mathrm{syn}})\ge 2(1-\widehat{F}(A,c))  \mathrm{dist} (T^{i-1}(\bp ), D_{\mathrm{syn}})$, where $\widehat{F}(A,c)\rightarrow 0$ (see below for definition), as $c\rightarrow 0$.

In fact, for $\bx=[x_1,\cdots, x_m]^\top$, we have
     $$(\mathrm{dist} (\bx, D_{\mathrm{syn}}))^2={(2m)^{-1}}\sum_{i\not=j}(x_i-x_j)^2=\sum_{i=1}^m x_i^2-m^{-1}(\sum_{i=1}^mx_i)^2= \bx^T  (I-m^{-1}E_0)  \bx ,$$
     where $E_0=\mathbf{e}\mathbf{e}^\top $ with $\mathbf{e}=[1,\cdots, 1]^T $. Then it holds that
     $$(\mathrm{dist} (T(\bx), D_{\mathrm{syn}}))^2=(T(\bx))^T  (I-m^{-1}E_0)  T(\bx)=4\bx^T  (I+cA)^T(I-m^{-1}E_0)(I+cA)  \bx.$$
     From the condition of $A$ it follows that $AE_0=E_0A=0$. It implies that
     $$(I+cA)^T(I-m^{-1}E_0)(I+cA)=I-m^{-1}E_0+2cA+c^2A^2.$$ It leads that
     $$(\mathrm{dist} (T(\bx), D_{\mathrm{syn}}))^2=4\bx^T  (I-m^{-1}E_0+2cA+c^2A^2)  \bx=4(\mathrm{dist} (\bx, D_{\mathrm{syn}}))^2+4\bx^T  (2cA+c^2A^2)  \bx.$$
     Denote $\bar{\bx}=\frac1m\sum_{i=1}^m x_i$. Then we have
     \begin{align*}
     &\bx^T  (2cA+c^2A^2)  \bx\\
     =&(\bar{\bx}\mathbf{e}+(\bx-\bar{\bx}\mathbf{e}))^T  (2cA+c^2A^2)  (\bar{\bx}\mathbf{e}+(\bx-\bar{\bx}\mathbf{e}))\\
     =&(\bar{\bx}\mathbf{e})^T  (2cA+c^2A^2)  (\bar{\bx}\mathbf{e})+2(\bar{\bx}\mathbf{e})^T  (2cA+c^2A^2)  (\bx-\bar{\bx}\mathbf{e})\\&+
      (\bx-\bar{\bx}\mathbf{e})^T  (2cA+c^2A^2)  (\bx-\bar{\bx}\mathbf{e})
     \end{align*}
        Again from the condition of $A$, we have $$(\bar{\bx}\mathbf{e})^T  (2cA+c^2A^2)  (\bar{\bx}\mathbf{e})=0,\qquad 2(\bar{\bx}\mathbf{e})^T  (2cA+c^2A^2)=0.$$ On the other hand, it is easily seen  that $(\mathrm{dist} (\bx, D_{\mathrm{syn}}))^2=\sum_{i=1}^m (x_i-\bar{\bx})^2=(\mathrm{dist} (\bx,\bar{\bx}\mathbf{e}))^2$. Then we have that
           \begin{align*}
     & | \bx^T  (2cA+c^2A^2)  \bx |
     \le (2c\|A\|+c^2\|A\|^2)(\mathrm{dist} (\bx,D_{\mathrm{syn}}))^2.
  \end{align*}
     It follows that
     \begin{align*}&(\mathrm{dist} (T(\bx), D_{\mathrm{syn}}))^2\ge 4(1-(2c\|A\|+c^2\|A\|^2))(\mathrm{dist} (\bx, D_{\mathrm{syn}}))^2\\
     \ge& \left(2(1-(2c\|A\|+c^2\|A\|^2)^{1/2})\right)^2(\mathrm{dist} (\bx, D_{\mathrm{syn}}))^2\\
     \equiv &(2(1-\widehat{F}(A,c)))^2(\mathrm{dist} (\bx, D_{\mathrm{syn}}))^2.\end{align*}
     Thus we obtain the claim. Consequently there must exists $l$ such that $T^{l}(\bx )\in B_{\gamma_0}$.

     Next we estimate the measure of the set $G_0$. Since the diameter of $\widetilde{G}_{\gamma_0}$ is small for small $\gamma_0$, from the expansivity of $T$, we obtain that each component $S$ of $T^{-i}(\widetilde{G}_{\gamma_0})$ possesses a small diameter for $i>0$. Thus
     we have that
    \begin{align*}
     {\rm Case {\rm (i)}.}\quad & \#^C( T^{-1}(S)\cap G_{\gamma_0})=2 \quad {\rm  \ \ if} \ \mathrm{dist} (S, (1,\cdots, 1))>3\gamma_0;\\
      {\rm Case {\rm (ii)}.}\quad & \#^C( T^{-1}(S)\cap G_{\gamma_0})\le 2^m \quad{\rm if} \ \mathrm{dist} (S, (1, \cdots, 1))\le 3\gamma_0,
      \end{align*}
          where $\#^C(S)$ denotes the number of connected components for a set $S$.

      On the other hand, the expansibility of $T$ implies that the measure of each component of $T^{-1}(S)$ is less than $(2(1-\widehat{F}(A,c)))^{-m}  \bm(S)$. Thus for case {\rm (i)} we have that $\bm(T^{-1}(S))\le 2^{-m+1}(1-\widehat{F}(A,c)))^{-m}  \bm(S)$ and for case {\rm (ii)} we have that $$\bm(T^{-1}(S))\le (1-\widehat{F}(A,c))^{-m}  \bm(S).$$

    %For each component $S$ of $T^{-i}\widetilde{G}_{\gamma_0}$, let $S_1$ be a component of $T^{-10}S$. Obviously, for small $\gamma_0$ we have that there is at most one number $1\le j\le 10$ such that ${\it dist}\{T^{j}S_1, (1,\cdots, 1)\}\le 3\gamma_0$. It implies that $T^{-10}S$ possesses at most $2^{m+9}$ components.

 Obviously, since $\gamma_0$ is small, we have that
     there is at most one number $j$ in $\{i,i+1,\cdots,i+10\}$ such that $\mathrm{dist} (T^{-j}(S), (1,\cdots, 1))\le 3\gamma_0$.
     Thus it is not difficult to see that for $l=10k+j$ with $1\le j< 10$, it holds that
     $
     \bm(\widetilde{G}_{-l})\le d_l \bm(\widetilde{G}_{\gamma_0})
     $
     with $d_l = 2^{(-m+1) (9k+j-1)}  (1-\widehat{F}(A,c))^{-ml}.$ Let $d_0=1$. For small $c$ we can easily see that $\sum_{l=0}^{\infty}d_l\le 4.$  Hence if $0\le \gamma_0<\frac{1}{2^{m+3}}$ and $c$ is small, we have
          \begin{align}\label{measure-G_0}
     \bm(G_0)\le& \bm(\cup_{l=0}^{\infty}T^{-l}(\widetilde{G}_{\gamma_0}))\le \sum_{l=0}^{\infty}d_l \bm(\widetilde{G}_{\gamma_0})\le 4\bm(\widetilde{G}_{\gamma_0})\nn\\
     \le &4\gamma_0  \bm(G_{\gamma_0})<\frac{1}{2^{m+1}}\bm(G_{\gamma_0}).
    \end{align}
     This completes the proof of the lemma.\hfill\qedbox
    \end{Proof}

The disordered part (\ref{disorder-part}) for the case of multi-node  can be easily obtained from the following corollary.

\begin{Corollary}
For almost each point $\bp $ in $G_{\gamma_0}$, there exists a finite time $l=l(\bp )$ such that
    $T^{l}(\bp )\in B_{\gamma_0}\cap B_0$, where $B_0$ is defined as in Lemma \ref{6.2}. \end{Corollary}

\begin{Proof}
     Assume the conclusion is not true. Then there exists a set $S\subset G_{\gamma_0}$ with a positive measure such that for each $l$ it holds that $T^l(S)\cap (B_{\gamma_0}\cup B_0)=\emptyset.$ We will prove that there exists a subset $S_0\subset S$ with a positive measure and $l=l(S_0)$ such that $T^l(S_0 ) \subset B_{\gamma_0} \cup B_0 $. From the contradiction, we end the proof.

    For this purpose, we claim that there exists a subset $\widetilde{S_0}\subset S$ and $l\in\NN$ such that $\bm(T^{l}(\widetilde{S_0}))\ge  \delta_1 \bm(G_{\gamma_0})$.

    In fact, from the claim and (\ref{measure-G_0}), we have that
    $$
    \frac{\bm(T^{l}(\widetilde{S_0})\cap B_0)}{\bm(T^{l}(\widetilde{S_0}))}\ge \frac{\bm(T^{l}(\widetilde{S_0}))-\bm(G_0)}{\bm(T^{l}(\widetilde{S_0}))}\ge 1-\frac{4\gamma_0}{ \delta_1}\ge 1/2.
    $$
   Define $S_0\subset \widetilde{S_0}$ such that $T^{l}(S_0)=T^{l}(\widetilde{S_0})\cap B_0\subset B_0$. Then $T^{l}(S_0)$ has a positive measure, which implies the measure of $S_0$ is positive. Thus the conclusion is obtained.

   Next we prove the existence of $\widetilde{S_0}$. Let $G_{\gamma_0}=G_1\cup G_2 \cup G_3$, where $G_1={G}_{\gamma_0}\cap D_{J_1}$, $G_2={G}_{\gamma_0}\cap D_{J_2}$ and $G_3=G_{\gamma_0}\backslash (G_1\cup G_2)$. It is clear that for a subset $S$ of any component of $G_3$, $T^j(S)$ has only one component for each $j=1, 2, \cdots, 10$. Moreover, the image of any subset of $G_1$ or $G_2$ under $T$ has at most $2^m$ components.

    Then for a subset $S$ of $G_1$ or $G_2$, $T(S)$ has at most $2^m$ components among which two ones lie in $G_1\cup G_2$ and the other $2^m-2$ ones lie in $G_3$. Subsequently, $T^2(S)$ has at most $(2^m-2)+2\cdot 2^m$ components, among which 4 components lie in $G_1\cup G_2$ and $3  (2^m-2)$ others lie in $G_3$. Similarly, we have that $T^3(S)$ has at most $(2^3-1)(2^m-2)+2^3$ components, among which $(2^3-1)(2^m-2)$ components lie in $G_3$ and $2^3$ others lie in $G_1\cup G_2$.

    On the other hand, $\bm(T^3(S))\ge 2^{3m}(1-F(A,c))^{3}  \bm(S)$, where $F(A,c)$ is defined in Proposition \ref{firststep-higher-dim}. Obviously $2^{3m}(1-F(A,c))^{3m}>(2^3-1)(2^m-2)+2^3$ for $m\ge 2$ and small $c$. Thus from Corollary \ref{Corollary}, we obtain the existence of $\widetilde{S_0}$. Thus the proof is complete.\hfill \qedbox
    \end{Proof}

\section{Conclusion}
The clustering phenomenon intermittent  behaviors have been widely found in weakly coupled map lattices by numerical experiments but without mathematical proof. Among these phenomenon, pseudo synchronization, i.e., successive transition between ordered
and disordered phases, is the most difficult from the point of view of mathematics. In this paper, we provide a complete proof for pseudo synchronization for weakly coupled tent-map lattices with arbitrarily many nodes. For weakly coupled piecewise-expanding map lattices with 2 nodes, we also obtain the same result.
How to extract more information on the dynamical properties by this work and previous results, for example, of G. Keller et al. \cite{Keller-Liverani1,Keller-Liverani0}, is one of our future interest.
  We will also be interested in the change of dynamical behavior when  a strong coupling decreases to zero.
In addition, we will study  the weakly coupled piecewise-expanding map lattices with arbitrarily many nodes in the future.

 \end{document}